\documentclass[a4paper,11pt]{amsart}
\usepackage[left=1in,right=1in,top=1in,bottom=1in]{geometry}
\usepackage[utf8]{inputenc}
\usepackage{amsmath}
\usepackage{amsfonts}
\usepackage{amscd}
\usepackage{amssymb,amsthm}
\usepackage{enumerate}
\usepackage{tikz,tikz-cd}
\theoremstyle{plain}
\newtheorem{thm}{Theorem}[section]

\newtheorem{prop}[thm]{Proposition}
\newtheorem{lem}[thm]{Lemma}
\newtheorem{lemma}[thm]{Lemma}
\newtheorem{cor}[thm]{Corollary}

\theoremstyle{definition}
\newtheorem{defn}[thm]{Definition}
\newtheorem{notation}[thm]{Notation}
\theoremstyle{remark}
\newtheorem{remark}[thm]{Remark}

\numberwithin{equation}{section}

\usepackage{tikz}
\usetikzlibrary{shapes,calc,arrows,intersections,decorations.pathreplacing,decorations.markings,shapes.geometric,calc,positioning,backgrounds}

\newcommand{\bC}{{\mathbb C}}

\newcommand{\bE}{{\mathbb E}}

\newcommand{\bN}{{\mathbb N}}

\newcommand{\bR}{{\mathbb R}}

\newcommand{\bZ}{{\mathbb Z}}

\newcommand{\cA}{{\mathcal A}}

\newcommand{\cC}{{\mathcal C}}

\newcommand{\cE}{{\mathcal E}}
\newcommand{\cF}{{\mathcal F}}
\newcommand{\cG}{{\mathcal G}}

\newcommand{\cN}{{\mathcal N}}

\newcommand{\cP}{{\mathcal P}}

\newcommand{\cS}{{\mathcal S}}

\newcommand{\cV}{{\mathcal V}}

\newcommand\id{\operatorname{id}}

\newcommand\ind{\operatorname{Ind}}

\DeclareMathOperator{\tr}{tr}
\DeclareMathOperator{\Tr}{Tr}

\DeclareMathOperator{\Comp}{Comp}

\newcommand{\norm}[1]{\left\|#1\right\|}
\newcommand{\set}[1]{\left\{#1\right\}}
\newcommand{\paren}[1]{\left(#1\right)}
\newcommand{\sq}[1]{\left[#1\right]}

\newcommand{\omegapic}{\pi_c}
\newcommand{\Tompic}{T/\omegapic}
\newcommand{\Tpic}{T_{\pi,c}}
\newcommand{\uu}[1]{{\underline{\underline{#1}}}}
\newcommand{\GCC}{\mathcal{GCC}}
\newcommand{\gps}[1][s]{T_{\pi,#1}}
\newcommand{\hsc}[1][c]{h^\pi_{#1\to s}}

\newcommand{\tdg}{test digraph}

\newcommand{\dg}{digraph}

\definecolor{green}{RGB}{13,177,75} %
\newcommand{\crr}{{\color{red}R}}
\newcommand{\cg}{{\color{green}G}}
\newcommand{\cb}{{\color{blue}B}}

\DeclareFontFamily{U}{mathb}{\hyphenchar\font45}
\DeclareFontShape{U}{mathb}{m}{n}{
      <5> <6> <7> <8> <9> <10> gen * mathb
      <10.95> mathb10 <12> <14.4> <17.28> <20.74> <24.88> mathb12
      }{} 
\DeclareSymbolFont{mathb}{U}{mathb}{m}{n}
\DeclareFontSubstitution{U}{mathb}{m}{n}
\DeclareMathSymbol{\Asterisk}      {2}{mathb}{"06}

\newcommand{\gp}{\mathop{\begin{tikzpicture}[baseline]
\node[circle, fill=cyan, draw=black, inner sep=0pt, minimum size=3pt] (mid) at (0, 2.5pt) {};
\foreach \x in {0,60,...,300} {
\node[circle, fill=cyan, draw, inner sep=0pt, minimum size=3pt] (\x) at ($(mid)!6pt!(mid.\x)$) {};
\draw (\x) -- (mid) ;
}
\end{tikzpicture}}}

\newcommand{\sstrc}{\begin{tikzpicture}
    \node [inner sep=0, minimum size=0] at (-6pt,0) {};
    \draw (-5pt,0) -- (0,0) node[draw, shade, circle, ball color=black!60!white, inner sep=0pt, minimum size=4pt] {};
    \node at (0,1pt) {};
\end{tikzpicture}}

\usepackage{hyperref}

\title{Random permutation matrix models for graph products}

\author{Ian Charlesworth}
\address{Cardiff University, School of Mathematics, Abacws, Senghennydd Rd, Cardiff CF24 4AG, United Kingdom}
\email{charlesworthi@cardiff.ac.uk}
\urladdr{https://www.ilcharle.com/}

\author{Rolando de Santiago}
\address{Department of Mathematics and Statistics, California State University, Long Beach \hfill\url{rolando.desantiago@csulb.edu}}
\email{rolando.desantiago@csulb.edu}

\author{Ben Hayes}
\address{Department of Mathematics,
University of Virginia, 141 Cabell Drive, Kerchof Hall, Charlottesville, VA, 22904}
\email{brh5c@virginia.edu}
\urladdr{https://sites.google.com/site/benhayeshomepage/home}

\author{David Jekel}
\address{\parbox{\linewidth}{Department of Mathematics, York University,}}
\email{david.jekel@gmail.com}
\urladdr{http://davidjekel.com}

\author{Srivatsav Kunnawalkam Elayavalli}
\address{\parbox{\linewidth}{Department of Mathematics, University of California, \\
San Diego, 9500 Gilman Drive \# 0112, La Jolla, CA 92093}}
\email{skunnawalkamelayaval@ucsd.edu}
\urladdr{https://sites.google.com/view/srivatsavke}

\author{Brent Nelson}
\address{Department of Mathematics, Michigan State University, 619 Red Cedar Road, C212 Wells Hall, East Lansing, MI 48824}
\email{brent@math.msu.edu}
\urladdr{https://users.math.msu.edu/users/banelson/}
\begin{document}

\maketitle

\begin{abstract}
Graph independence (also known as $\epsilon$-independence or $\Lambda$-independence) is a mixture of classical independence and free independence corresponding to graph products of groups or operator algebras.  Using conjugation by certain random permutation matrices, we construct random matrix models for graph independence with amalgamation over the diagonal matrices.  This yields a new probabilistic proof that graph products of sofic groups are sofic.
\end{abstract}

\section{Introduction}

\subsection{Motivation}
Non-commutative probability theory views operators acting on a Hilbert space as analogous to random variables acting by multiplication on $L^2$.
In this setting there are five natural notions of independence: classical (or tensor) independence, free independence, (anti-)monotone independence, and boolean independence \cite{MR2016316, MR1895232}.
There are various generalizations of these notions; of particular interest is the setting for families of algebras where the independence relation is not necessarily assumed to be the same between any given pair \cite{BW1998,Mlot2004,Wysoczanski2010,KW2013,SpWy2016,JekelLiu2020,AMVB2023}.
These types of independence arise in surprisingly diverse contexts.
 Free independence was motivated by the study of free products of $\mathrm{C}^*$ and von Neumann algebras \cite{Avitzour1982,Voiculescu1985}.
The study of interacting quantum particles on various Fock spaces \cite[\S X.7]{ReedSimon1972} \cite{Sharan2023} led to the study of non-commutative Brownian motions \cite{BS1996,Lu1997} and from there to the abstract notions of independence.
Non-commutative independence was found to describe the large $N$ behavior of various types of random matrix models \cite{VoicAsyFree,CHS2018,male2020traffic}, resulting in a fruitful interaction between probability, combinatorics, algebra, and operator algebras.
Non-commutative independences correspond to various product operations, not only on algebras, but also on graphs, which has led to various applications of non-commutative probability and random matrix theory to the spectral theory of graphs \cite{NBGO2004,HO2007,BCRandomPerm,GVK2023,bordenave2023norm}.
%

Our paper focuses on \emph{graph independence}.
Just as free independence was inspired by free products of groups and describes the distribution of elements of a free product, graph independence is inspired by the graph products of groups developed by Green \cite{Gr90} and describes the distribution of elements of a graph product.
Given a graph $\mathcal{G} = (\mathcal{V},\mathcal{E})$ and groups $\Gamma_v$ for $v \in \mathcal{V}$, the graph product of groups is $\Gamma$ given by quotienting $*_{v \in V(\Gamma)} \Gamma_v$ by the relations $[g,h] = 0$ when $g \in \Gamma_v$ and $h \in \Gamma_w$ and $v \sim w$.
The analogous constructions for $\mathrm{C}^*$-algebras and von Neumann algebras were introduced by M\l{}otkowski and later studied by Caspers and Fima \cite{Mlot2004,CaFi17}.  
(We caution that graph independence has also been variously called $\epsilon$-independence and $\Lambda$-independence \cite{Mlot2004,SpWy2016} in the literature.)


Our goal is to construct matrix models for graph independence using random permutation matrices, analogous to models studied in the unitary case by Charlesworth and Collins \cite{CC2021}, and also with different terminology by Morampudi and Laumann \cite{ML2019} (see \cite{mageethomasstrongly} for strong convergence results of matrix models for right-angled Artin groups, which are generated by graph independent Haar unitaries). Random permutation matrices are an object of great interest in random matrix theory, and most questions that have been studied for Haar unitary matrices have an analog for permutations.  Analogous to Voiculescu's theorem that independent random Haar unitaries are asymptotically freely independent \cite{VoicAsyFree}, Nica showed that the same is true for independent random permutation matrices \cite{Nica1993}.  Strong convergence (i.e.\ convergence of operator norms of polynomials in our matrices) was studied both in the unitary case \cite{MaleCollins, ParraudHaar, ParraudSmooth} and the permutation case \cite{BCRandomPerm, bordenave2023norm}. Finite free convolution can be defined using unitary conjugation or conjugation by permutation matrices \cite[Section 2]{MSSFiniteFree}, in a manner analogous to free convolution \cite[Theorem 4.1]{CDAmalgam}.

Au, C{\'e}bron, Dahlqvist, Gabriel, and Male showed that conjugation by random permutation matrices  asymptotically produces free independence with amalgamation over the diagonal subalgebra \cite{ACDGM2021}.
Their argument uses traffic spaces, a combinatorial framework developed by Male for the study of permutation-invariant random matrix models which also naturally encompasses tensor, free, and boolean independence \cite{male2020traffic}.
We will use traffic moments to study permutation models for graph independence, analogous to models studied by \cite{ML2019,CC2021} in the unitary case.
Given matrix approximations for the algebras $(A_v)_{v \in \mathcal{G}}$, we construct matrix approximations for the graph product with amalgamation over the diagonal by conjugating the matrices for $A_v$ by an appropriate random permutation.

One interest of the permutation construction is that it preserves finer algebraic or combinatorial properties of matrices that might be destroyed by unitary conjugation.
For instance, if the matrix approximations for the generators of $A_v$ have integer entries or are permutation matrices, then so will the matrix approximations for the generators of the graph product.
This has applications in the setting of sofic groups, which are those that can be simulated or approximated by permutations in a suitable sense (see \S \ref{subsec: sofic}).
In \cite{graphproductofsoficissofic}, it was shown that sofic groups are preserved under graph products, by studying group actions on graphs.
Our results on random permutations give a new probabilistic proof that soficity is preserved by graph products.

\subsection{Main results} \label{subsec: results}

We now describe the construction of the permutation matrix model for graph independence.  Since we will attempt to model graph products, we will need to force certain matrices to commute with each other, and certain matrices to be asymptotically free.
As in \cite{CC2021}, we will accomplish this by taking the models in a tensor product of several copies of $M_N(\bC)$, with matrices having only scalar components in certain tensor factors; in this way we can ensure that matrices which are meant to commute do so.

Heuristically, the index set of this tensor product will be a finite set of \emph{strings} $\cS$.  Given a subalgebra of this larger product formed by replacing some of the tensor factors with copies of $\bC 1_N$, we will think of its elements as corresponding to collections of beads on the strings where the algebra has a non-trivial factor.  Two algebras commute, then, if the beads representing their elements can slide past each other on this collection of strings.
For more detail on this picture, refer to \cite[\S3.2]{CC2021} or more generally \cite{MR2651902}.

The vertex set in a graph product we will think of as a set of \emph{colors}, so matrices will have a color based on which vertex algebra they are approximating an element of.
We specifically use the language of ``colors'' because later in the argument there will be additional graphs with other vertex sets, with edge labelings coming from the algebras.
The information of which tensor factors of a matrix are allowed to be non-scalar is determined by its color.

We will choose our set of strings and the assignments of colors to sets of strings in such a way that matrices will share a string in common precisely when the graph product structure insists that the algebras they are modelling should be freely independent.
Given a prescribed finite graph $\Gamma$ it is always possible to choose $\cS$ and $\sstrc$ with this property; one approach was outlined in \cite[Section 3.1]{CC2021}.

It will be convenient to implicitly attach the index to each factor and suppress the isomorphisms required to permute the elements, and moreover to make the inclusion of a tensor product over a smaller set into that of a larger set implicit by adding several copies of $\bC$.  To wit, given finite sets $S \subseteq T$ we will often take the embedding
\[
\bigotimes_S M_N(\bC) \cong \paren{\bigotimes_S M_N(\bC)} \otimes \paren{\bigotimes_{T \setminus S} 1_N\bC} \subseteq \bigotimes_T M_N(\bC).\]
After \cite{CC2021}, we will adopt the following terminology:
\begin{itemize}
	\item $\mathcal{S}$ is a finite set whose elements are called \emph{strings};
	\item $\mathcal{C}$ is a finite set whose elements are called \emph{colors};
	\item we are given a relation $\sstrc$ between $\mathcal{S}$ and $\mathcal{C}$;
  \item for $c \in \mathcal{C}$, let $\mathcal{S}_c = \{s: s \sstrc c\}$ and for $s \in \mathcal{S}$, let $\mathcal{C}_s = \{c: s \sstrc c\}$; we assume that $\cS_c \neq \emptyset$ for each $c \in \cC$.
\end{itemize}
We will have that a random matrix colored by $c \in \cC$ will be valued in \[\bigotimes_{\cS_c} M_{N}(\bC) \otimes \bigotimes_{\cS\setminus \cS_c} 1_N\bC.\]
The notation $\sstrc$ is meant to be suggestive: we can think of each tensor factor corresponding to a horizontal strings, and matrices as occupying some subset of these strings (the ones corresponding to $\cS_c$).
For a more explicit picture, see \cite[\S3.2]{CC2021}. Given a simple graph $\cG=(\cC,\cE)$, we say that $\chi\colon [k]\to \cC$ is a \textbf{$\cG$-reduced word} if whenever $1\leq i<j<k$ and $\chi(i)=\chi(j)$, there is an $\ell\in \{i+1,\cdots,j-1\}$ with $(\chi(i),\chi(\ell))\notin \cE$. 

We are now ready to state the main theorem:

\begin{thm} \label{thm: permutation model}
Let $\cG = (\cC, \cE)$ be a simple graph with vertex set $\cC$, $\cS$ be a finite set, and $\sstrc$ be as above so that $\cS_c \cap \cS_{c'} = \varnothing$ if and only if
$(c, c') \in \cE$.

For $N \in \bN$, let $\Delta_{N^{\#\cS}}$ be the conditional expectation onto the diagonal $*$-subalgebra $D_N$ of $\bigotimes_{\cS}M_N(\bC)$.

Let $\chi : [k] \to \cC$ be such that $\chi(1)\cdots\chi(k)$ is a $\cG$-reduced word, $\ell : [k] \to \bN$, and for $i = 1, \ldots, k$, $j = 1, \ldots, \ell(k)$, and $N \in \bN$, let $\Lambda_{i,j}^{(N)} \in D_N$ and $X_{i,j}^{(N)} \in \bigotimes_{S_{\chi(i)}} M_N(\bC)$ be deterministic, with $\sup_{N,i,j} \norm{\Lambda_{i,j}^{(N)}} < \infty$ and $\sup_{N,i,j} \norm{X_{i,j}^{(N)}} < \infty$.

Further, let $\set{\Sigma_{c}^{(N)} : c \in \cC}$ be a family of independent uniformly random permutation matrices, with $\Sigma_c \in \bigotimes_{\cS_c} M_N(\bC)$, and write
\[ \uu{X}_{i,j}^{(N)} = \paren{\Sigma_{\chi(i)}^{(N)}}^* X_{i,j}^{(N)} \Sigma_{\chi(i)}^{(N)} \otimes I_N^{\otimes\cS\setminus\cS_{\chi(i)}} \in \bigotimes_{\cS} M_N(\bC). \]

Then, setting
\[Y_i^{(N)} = \Lambda_{i,1}^{(N)} \uu{X}_{i,1}^{(N)} \cdots \Lambda_{i,\ell(i)}^{(N)} \uu{X}_{i,\ell(i)}^{(N)} \in \bigotimes_{s \in \cS} M_N(\bC),\]
we have
\begin{equation}\label{eqn: graph ind ovre diagonal intro}
  \lim_{N \to \infty} \norm{\Delta_{N^{\#\cS}}[(Y_1^{(N)} - \Delta_{N^{\#\cS}}[Y_1^{(N)}]) \dots (Y_k^{(N)} - \Delta_{N^{\#\cS}}[Y_k^{(N)}])]}_2 = 0 \text{ almost surely.}
  \end{equation}
\end{thm}

Suppose we are given algebras $A_{c},c\in \cC$, each unitally embedded inside an algebra $\cA$ equipped with a linear functional $\tau$. Then graph independence means that for every $\cG$-reduced word $\chi\colon [k]\to \mathcal{C}$ and every tuple $(a_{c})_{c}\in \prod_{c\in \cC}A_{c}$ with $\tau(a_{c})=0$ for all $c\in \cC$ we have
\[\tau(a_{\chi(1)}\cdots a_{\chi(k)})=0.\]
Equation (\ref{eqn: graph ind ovre diagonal intro}) can thus be interpreted as ``asymptotic graph independence over the diagonal". This type of relative result is necessary for our theorem, since conjugation by permutations leaves the diagonal invariant. Thus, as soon as the graph is not complete, it is impossible for conjugation by permutations (random or deterministic) to create asymptotic graph independence. Conditional independence appears often in the classical setting, and free independence with amalgamation (also called ``operator-valued freeness") also has numerous applications in random matrices \cite{BDJ2008, RadulescuIFGF, CLGMAsyFree}, operator algebras \cite{ BDJ2008,chifan2022nonelementarily, CartanAFP, IPP, RadulescuIFGF, Voiculescu1985}, and graph theory \cite{GVK2023}. This appearance of ``graph independence with amalgamation" is surprising, as it does not have a group analog (unless the amalgam is central in each of the vertex groups). Nevertheless, our theorem suggests the possibility of the theory of ``graph independence with amalgamation", and is the first instance where such a notion has been suggested in the literature.
We leave the development of such a theory as a subject of future research. We expect that once the correct definition is worked out, it will have the same exciting applications to random matrices, operator algebras, and graph theory that conditional independence and free independence with amalgamation do.

The proof of Theorem \ref{thm: permutation model} draws ideas from the tensor-product matrix models of \cite{CC2021} and the traffic moment techniques of \cite{ACDGM2021}. In fact, the model in \cite{CC2021} is precisely the model in Theorem \ref{thm: permutation model} with random unitaries (distributed according to Haar measure) instead of uniformly random permutations. Note that if $k\in \bN$, $R\in [0,+\infty]$ and $C\in M_{k}(\bC)$ with $\|C\|\leq R$, then for any $p\in [1,2],q\in [2,+\infty)$ we have:
\[\|C\|_{2}^{2/p}R^{1-\frac{2}{p}}\leq \|C\|_{p}\leq \|C\|_{2}\leq \|C\|_{q}\leq R^{1-\frac{2}{q}}\|C\|_{2}^{2/q}.\]
Hence Theorem \ref{thm: permutation model} recovers \cite[Theorem 1.3]{ACDGM2021}, by considering a graph with no edges.

Furthermore, we use this result to give another proof that the graph product of sofic groups is sofic.  Sofic groups, defined by Gromov \cite{GromovSurjun} and named by Benjy Weiss \cite{WeissSofic}, are groups that can be approximated by permutations.  Suppose $\Gamma$ is a countable group generated by $(g_j)_{j \in \bN}$, and write $g_{-j} = g_j^{-1}$ for notational convenience.  Then $\Gamma$ is sofic if and only if there exist $N \times N$ permutation matrices $T_j^{(N)}$ such that, setting $T_{-j}^{(N)} = (T_j^{(N)})^*$, we have
\[
\lim_{N \to \infty} \tr_N(T_{j(1)}^{(N)} \dots T_{j(m)}^{(N)}) = \delta_{g_{j(1)} \dots g_{j(m)}=e}
\]
for all $m \in \bN$ and $j(1)$, \dots, $j(m) \in \bN \cup (-\bN)$ (for proof of this characterization, see Proposition \ref{prop: sofic}).  For example, one can deduce that the free group $F_d$ is sofic from Nica's theorem that Haar random unitaries $T_1^{(N)}$, \dots, $T_d^{(N)}$ are almost surely asymptotically free \cite[\S 4.1]{Nica1993}. Related to soficity, we give an application of Theorem \ref{thm: permutation model} to the entropy theory of graph product von Neumann algebras in \cite{AIMSuperTeamII}, which also relies on a modified version of soficity suitable for tracial $*$-algebras.

Sofic groups themselves have seen numerous applications in recent years: their Bernoulli shift actions can be completely classified (by \cite{Bow, BowenOrn, SewardOrn}); they are known to satisfy the determinant conjecture \cite{ElekSzaboDeterminant}(a conjecture arising in the theory of $L^{2}$-invariants) and consequently their $L^{2}$-torsion is well-defined \cite{LuckBook}; they are known to admit a version of L\"{u}ck approximation \cite{ElekSzaboDeterminant}; they satisfy Gottschalk's surjunctivity conjecture \cite{GromovSurjun}; and they are known to satisfy Kaplansky's direct finiteness conjecture \cite{ElekSzaboDF}.  In fact, any group for which one of these properties is known is also known to be sofic; it is a large open question whether or not every group is sofic. We refer the reader to \cite{LewisICM} for further applications of sofic groups, particularly to ergodic theory.

The fact that graph products of sofic groups are sofic was shown by \cite{graphproductofsoficissofic} using a construction of group actions on graphs.  Our Theorem \ref{thm: permutation model} leads to a new probabilistic proof of this fact, in the spirit of \cite{Nica1993} and \cite{ACDGM2021}.

\begin{prop} \label{prop: sofic graph product}
Let $\mathcal{G} = (\mathcal{C},\mathcal{E})$ be a finite graph and let $(G_c)_{c \in \mathcal{C}}$ be countable sofic groups.  Then $\gp_{c \in \mathcal{G}} G_c$ is sofic.
\end{prop}

We remark that a key trick for the proof is that for a non-identity word $g_{j(1)} \dots g_{j(m)}$, the permutation matrix satisfies
\[
\norm{\Delta(T_{j(1)}^{(N)} \dots T_{j(m)}^{(N)})}_2 \to 0
\]
because $\Delta(T_{j(1)}^{(N)} \dots T_{j(m)}^{(N)})$ is diagonal projection matrix whose trace tends to zero.  This is why the graph independence over the diagonal in Theorem \ref{thm: permutation model} reduces to plain graph independence in this case.

\textbf{Acknowledgements}.  We thank Dimitri Shlyakhtenko for his mentorship and lively discussions; and we thank IPAM and the Lake Arrowhead Conference Center for hosting the Quantitative Linear Algebra long program second reunion conference, where some of these discussions took place.  We thank the American Institute of Mathematics SQuaRES program for hosting us for a week each in April 2022 and April 2023 to collaborate on this project. BH was supported by the NSF grant DMS-2000105. IC was supported by long term structural funding in the form of a Methusalem grant from the Flemish Government. DJ was supported by postdoctoral fellowship from the National Science Foundation (DMS-2002826). SKE was supported by a Simons postdoctoral fellowship. BN was supported by NSF grant DMS-1856683.

\section{Preliminaries}

\subsection{Combinatorial background}
We collect here several notions from combinatorics which will be of use later in the paper.

Let $S$ be a finite set. A \emph{partition} on $S$ is a collection $\set{B_1, \ldots, B_k}$ of non-empty subsets of $S$ which are disjoint and whose union is $S$: that is, $S = B_1 \sqcup \cdots \sqcup B_k$.
The sets $B_i$ are called the \emph{blocks} of $\pi$.

A partition $\pi$ induces an equivalence relation on $S$.
For $s, t \in S$ we write $s \sim_\pi t$ if and only if $s$ and $t$ are in the same block of $\pi$ (i.e., there is $B \in \pi$ with $s, t \in B$).

The set of partitions of $S$ is denoted by $\cP(S)$.
It is partially ordered by reverse refinement: namely, given $\pi, \rho \in \cP(S)$, we write $\pi \leq \rho$ if and only if every block of $\rho$ is a union of blocks of $\pi$.
This order makes $\cP(S)$ into a lattice, and so for partitions $\pi, \sigma \in \cP(S)$ we have partitions $\pi \vee \sigma$ and $\pi \wedge \sigma$, their join and meet respectively.

For $N \in \bN$, we write $[N] = \set{1, \ldots, N}$.

Graphs make an appearance in our work in two distinct senses: as index sets to graph products, and as labels of traffic moments.
We will draw an explicit distinction between these two types of objects in an attempt to keep the contexts clear.

\begin{defn}
A \emph{(simple undirected) graph} is a pair $\cG = (\cV, \cE)$ where $\cV$ is a finite set consisting of \emph{vertices}, and $\cE \subset \cV \times \cV$ is a set of \emph{edges}.
We insist that $\cE$ is symmetric (i.e., $(x, y) \in \cE$ if and only if $(y, x) \in \cE$) and non-reflexive (i.e., $(x, x) \notin \cE$ for any $x \in \cV$; that is, we do not allow self-loops).

An \emph{(undirected) multigraph} is the same except $\cE$ is a multiset, allowing parallel edges.
\end{defn}

\begin{defn}\label{defn:digraph}
A \emph{digraph} (or \emph{directed (multi-)graph}) is a tuple $G = (V, E, \cdot_+, \cdot_-)$ where $V$ is a finite set of \emph{vertices}, $E$ is a finite set of \emph{edges}, and $\cdot_-, \cdot_+ : E \to V$ are functions specifying the \emph{source} and \emph{target} of each edge.
Note that we allow parallel edges and self-loops.

Given a digraph $G$, we will write $V(G)$ and $E(G)$ for its vertex and edge sets, respectively.
\end{defn}

\subsection{Graph products}

We first explain graph products of groups, both as background for our application to sofic groups and also motivation for the graph products of von Neumann algebras.  As the vertices of the graph will label distinct groups or algebras, we will often refer to them as ``colors'', saving ``vertices'' for the objects in digraphs.

Let $\mathcal{G} = (\cC,\cE)$ be a finite graph, and let $(G_c)_{c \in \cC}$ be groups.  The \emph{graph product} $\Gamma = \gp_{c \in \mathcal{G}} \Gamma_c$ is defined as the free product of $(\Gamma_c)_{c \in \cC}$ modulo the normal subgroup generated by commutators $[g,h]$ for $g \in \Gamma_c$ and $h \in G_w$ for adjacent pairs of vertices $(v,w)$ in the graph $\Gamma$.  Every element of $\gp_{v \in \Gamma} \Gamma_c$ can be expressed as some product
\[
g_1 \dots g_m,
\]
where $g_j \in \Gamma_{\chi(j)}$ for $j = 1$, \dots, $m$.  Two words yield the same element of $\Gamma$ if and only if the one word can be transformed into the other by the following operations:
\begin{itemize}
    \item Combining two letters $g_i$ and $g_{i+1}$ into a single element $(g_i g_{i+1})$ if $\chi(i) = \chi(i+1)$.
    \item Conversely, replacing a single letter $g_i$ by two letters $g_i'$, $g_i''$ such that $g_i' g_i'' = g_i$ in $\Gamma_{v(i)}$.
    \item Removal or insertion of an identity letter $e$.
    \item Permuting two adjacent letters $g_i$ and $g_{i+1}$ if $\chi(i) \sim \chi(i+1)$ in $\mathcal{G}$.
\end{itemize}
The operations of commutation, combining letters, and removing identity letters can be used to simplify a word as much as possible.  A word is called $\mathcal{G}$-\emph{reduced} if there are no identity letters and it is not possible to perform commutations so that two letters can be combined, or equivalently, between any two letters $g_i$ and $g_j$ with $\chi(i) = \chi(j)$, there exists some $g_k$ with $\chi(k)$ not equal or adjacent to $v(i)$ (hence, it not possible to combine $g_i$ and $g_j$ together since we cannot commute $g_i$ past $g_k$).  Every word is equivalent to a reduced word, and two reduced words represent the same group element if and only if they are equivalent by commutation operations (see the discussion following Definition 3.5 of \cite{Gr90}).

We will now define the graph product of von Neumann algebras, and some important related notions.
We will have a graph $\cG = (\cC, \cE)$ and a collection of finite tracial von Neumann algebras $(M_c, \tau_c)$ for each $c \in \cC$.
The graph product will be constructed as finite von Neumann algebra containing a copy of each $M_c$ in such a way that $M_c$ and $M_{c'}$ are in tensor product position if $(c, c') \in \cE$, and in free position otherwise.

\begin{defn}
Let $\cG = (\cC, \cE)$ be a graph.
We say that a word $(c_1, \ldots, c_n) \in \cC^n$ is \emph{$\cG$-reduced} provided that whenever $i < k$ are such that $c_i = c_k$, there is some $j$ with $i < j < k$ so that $(c_i, c_j) \notin \cE$.
\end{defn}

If $(c_1, \ldots, c_n) \in \cC^n$ is such a word and $x_i \in M_{c_i}$, then saying the word is \emph{not} $\cG$-reduced is exactly saying that two $x_i$'s from the same algebra could be permuted next to each other and multiplied, giving a shorter word.

\begin{defn}
Let $\cG = (\cC, \cE)$ be a graph, $(M, \tau)$ be a tracial von Neumann algebra, and for each $c \in \cC$ let $M_c \subseteq M$ be a von Neumann sub-algebra.
Then the algebras $(M_c)_{c \in \cC}$ are said to be \emph{$\cG$-independent} if: $M_c$ and $M_d$ commute whenever $(c, d) \in \cE$; and whenever $(c_1, \ldots, c_n)$ is a $\cG$-reduced word and $x_1, \ldots, x_n \in M$ are such that $x_i \in M_{c_i}$ and $\tau(x_i) = 0$, we have $\tau(x_1\cdots x_n) = 0$.

Conversely, given a collection $(M_c, \tau_c)$ of tracial von Neumann algebras, their \emph{graph product} $(M, \tau) = \gp_{c \in \cG} (M_c, \tau_c)$ is the von Neumann algebra generated by copies of each $M_c$ which are $\cG$-independent, so that $\tau|_{M_c} = \tau_c$.
(That the graph product exists and is unique was shown in \cite{Mlot2004}).

When the trace is clear from context, we may write simply $M = \gp_{c \in \cG} M_c$.
\end{defn}

Notice that if $\cG = (\cC, \emptyset)$ then $\gp_{c \in \cG} M_c = \Asterisk_{c \in \cC} M_c$; on the other hand, if $\cG$ is a complete graph, then $\gp_{c \in \cG} M_c = \bigotimes_{c \in \cC} M_c$.

\section{Proof of the main result}
\label{sec:permutationtraffic}

This section will present the proof of Theorem \ref{thm: permutation model}.
Our strategy will be to first establish that the convergence in \eqref{eqn: graph ind ovre diagonal intro} holds in $L^2$ (and hence in probability).
Having done so, we will use concentration estimates to upgrade to almost sure convergence.

We use the same notation introduced in \S \ref{subsec: results}.  However, throughout most of this section, we will suppress the explicit dependence upon $N$ and write merely $\uu{X}_{i,j}$, $\Sigma_{i,j}$, or $\Delta$, for example.
Combinatorial arguments are prone to intricate notation, and while this paper is no exception, we promise the reader that readability was foremost in our minds when deciding whether or not to introduce additional notation. Nevertheless, should the reader find themselves overwhelmed by the notation at any point, we invite them to consult Appendix~\ref{apndx: notation} for a summary table.  The appendix also includes worked visual examples of the constructions used in the proof.

\subsection{Traffic moments} \label{subsec:trafficmoments}

Similar to \cite{ACDGM2021}, the proof of Theorem~\ref{thm: permutation model} will use ideas from Camille Male's traffic independence framework \cite{male2020traffic} which was introduced to describe the behaviour of random matrix models which are invariant under conjugation by permutation matrices.  Therefore, we introduce the notation of traffic moments here, focusing on the case of matrices since we do not need abstract traffic spaces here.

\begin{defn}
  A \emph{\tdg{} over $M_N(\bC)$} is a finite (multi)\dg{} with a labeling of the edges by elements of $M_N(\bC)$: that is,
\begin{itemize}
    \item a finite set $V$ of vertices;
    \item a finite set $E$ of oriented edges;
    \item two maps $E \to V, e \mapsto e_-$ and $E \to V, e \mapsto e_+$ specifying the source and target of the edge $e$, respectively;
    \item a labeling $E \to M_N(\bC), e \mapsto X_e$ assigning a matrix to each oriented edge.
\end{itemize}

Continuing our convention from Definition~\ref{defn:digraph}, given a \tdg{} $T$ we will write $V(T)$ and $E(T)$ for its vertex and edge sets, respectively.
\end{defn}

\begin{remark}
  It is more common to refer to these objects simply as ``test graphs''.
  We emphasize ``\dg{}s'' to make clear the distinction from the (simple) graphs in this paper which index graph products.
  In particular, \dg{}s are directed and allow parallel edges and self-loops.
\end{remark}

\begin{notation}
\label{not: components and quotients}
  For a \dg{} $T$, we denote by $\Comp(T)$ the set of weakly connected components of $T$ (i.e., the components of the undirection of $T$).
  We will also write $\Comp(T)$ for the components of a \tdg{} $T$, where each component can be viewed as a \tdg{} by restricting the labeling $e \mapsto X_e$.

  Given a partition $\pi$ of the vertices of a \dg{} $T$, we will write $T/\pi$ for the \dg{} induced by identifying vertices according to $\pi$.
  (Specifically, $V(T/\pi) = \pi$, $E(T/\pi) = E(T)$, and the source and range maps become $e \mapsto [e_\pm]_{\pi}$.)
  If $T$ is a \tdg{}, we will treat $T/\pi$ as a \tdg{} with the same edge labels.
\end{notation}

\begin{defn}
\label{defn:testtrace}
  For a \tdg{} $T$ over $M_N(\bC)$, the \emph{trace} $\tau_N(T)$ is defined as
\[
\tau_N(T) = \frac{1}{N^{\# \Comp(T)}} \sum_{i: V(T) \to [N]} \prod_{e \in E(T)} (X_e)_{i(e_+),i(e_-)}.
\]
Meanwhile, the \emph{injective trace} is defined as
\[
\tau_N^\circ(T) = \frac{1}{N^{\# \Comp(T)}} \sum_{i: V(T) \hookrightarrow [N]} \prod_{e \in E(T)} (X_e)_{i(e_+),i(e_-)};
\]
here the sum is only over injective maps $i: V(T) \hookrightarrow [N]$.
\end{defn}

To prove that the quantity in Theorem \ref{thm: permutation model} converges to zero in $L^2$, we must show that the expectation of its square goes to zero.  Thus, to prove the theorem, we will study the expectation of the trace of various combinations of matrices using the operations of addition, multiplication, and conditional expectation onto the diagonal.  Such expressions can be realized by traffic moments, as we will illustrate with a few examples.

For matrices $A$, $B$, $C$, and $D$ the trace $\tr(ABCD)$ would be represented as the trace of the following \tdg{}
\begin{center}
\begin{tikzpicture}
    \node[circle,fill,label= above right: $0$] (0) at (2,2) {};
    \node[circle,fill,label=above left: $1$] (1) at (0,2) {};
    \node[circle,fill,label=below left: $2$] (2) at (0,0) {};
    \node[circle,fill,label=below right: $3$] (3) at (2,0) {};

    \draw[<-] (0) to node[auto,swap] {$A$} (1);
    \draw[<-] (1) to node[auto,swap] {$B$} (2);
    \draw[<-] (2) to node[auto,swap] {$C$} (3);
    \draw[<-] (3) to node[auto,swap] {$D$} (0);
\end{tikzpicture}
\end{center}
Next, to express $\tr(A \Delta(BCD)) = \tr(\Delta(A) BCD) = \tr(\Delta(A) \Delta(BCD))$, we would use the \tdg{}
\begin{center}
\begin{tikzpicture}
    \node[circle,fill,label=left: $1$] (1) at (1,1.5) {};
    \node[circle,fill,label=below left: $2$] (2) at (0,0) {};
    \node[circle,fill,label=below right: $3$] (3) at (2,0) {};

    \path[<-] (1) edge [out=60,in=120,looseness=20] node[above] {$A$} (1);
    \path[<-] (1) edge node[auto,swap] {$B$} (2);
    \path[<-] (2) edge node[auto,swap] {$C$} (3);
    \path[<-] (3) edge node[auto,swap] {$D$} (1);
\end{tikzpicture}
\end{center}
By identifying the vertices $0$ and $1$ together, we force the indices for the input and output of the edge for $A$ to agree, hence taking the diagonal of $A$.  Next, consider $\tr(\Delta(A) \Delta(B \Delta(C) D))$.  In this case, we would identify the two endpoints of the edge for $C$, and then identify the start point of the edge for $B$ with the endpoint of the edge $D$, resulting in the following \tdg{}
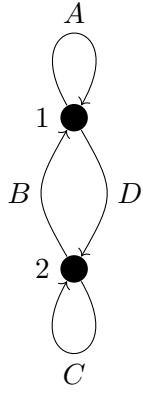
\begin{center}
\begin{tikzpicture}
    \node[circle,fill,label=left: $1$] (1) at (0,2) {};
    \node[circle,fill,label=left: $2$] (2) at (0,0) {};

    \path[<-] (1) edge[out=60,in=120,looseness=20] node[above] {$A$} (1);
    \path[<-] (1) edge[out=240, in=120, looseness=1.4] node[left] {$B$} (2);
    \path[<-] (2) edge[out=240, in=300, looseness=20] node[below] {$C$} (2);
    \path[<-] (2) edge[out=60, in=300, looseness=1.4] node[right] {$D$} (1);
\end{tikzpicture}
\end{center}

We will leave the full details to the later proof, but the simpler case $\tr((Z_1 - \Delta(Z_1)) \dots (Z_k - \Delta(Z_k)))$, as in \cite[\S 4.2]{ACDGM2021}, gives a sense of the method.  Expanding this expression by multilinearity yields terms indexed by subsets $I \subseteq [k]$, where $I$ represents the set of indices $j$ where the $-\Delta(Z_j)$ term is chosen and the $[k] \setminus I$ is the set of indices where the term $Z_j$ is chosen.  One thus obtains the trace of a product of $Z_j$'s and $\Delta(Z_j)$'s times $(-1)^{|I|}$.  The trace of the product can be represented as the trace of a \tdg{} as follows.  Let $T$ be a $k$-cycle \dg{} with edges labelled by $Z_1$, \dots, $Z_k$.  Given $I \subseteq [k]$, identify the start and end vertices of the edge for $A_j$ for each $j \in I$, and let $T_I$ be the quotient \dg{}.  This has the effect of replacing $Z_j$ with $\Delta(Z_j)$.  One thus obtains
\[
\tr((Z_1 - \Delta(Z_1)) \dots (Z_k - \Delta(Z_k))) = \sum_{I \subseteq [k]} (-1)^{|I|} \tau[T_I].
\]
In \cite{ACDGM2021}, the authors use a combinatorial formula for $\tau(T_I)$ in terms of the injective trace and then exploit cancellation.  Our goal in the next section will be to find such a combinatorial formula for the color/string model of graph products.

Another wrinkle in the case of Theorem \ref{thm: permutation model} is that we consider monomials with coefficients in the diagonal matrices, expressions such as $\tr(\Lambda_1 Z_1 \dots \Lambda_m Z_m)$, where $Z_1$, \dots, $Z_m$ are the random matrices in our model and $\Lambda_1$, \dots, $\Lambda_m$ are deterministic diagonal matrices.  Since $\Delta(\Lambda_j) = \Lambda_j$, we can put the matrices $\Lambda_j$ on self-looping edges.
Thus, for example, the expression $\tr(\Lambda_1 Z_1 \dots \Lambda_5 Z_5)$ would be represented by the \tdg{}
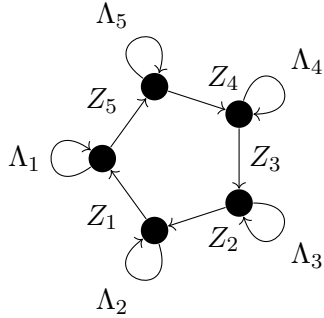
\begin{center}
    \begin{tikzpicture}
        \node[circle,fill] (0) at (180:1) {};
        \node[circle,fill] (1) at (252:1) {};
        \node[circle,fill] (2) at (324:1) {};
        \node[circle,fill] (3) at (36:1) {};
        \node[circle,fill] (4) at (108:1) {};

        \path[<-] (0) edge[out=144, in=216, looseness = 10] node[auto,swap] {$\Lambda_1$} (0);
        \path[<-] (0) edge node[auto,swap] {$Z_1$} (1);
        \path[<-] (1) edge[out=216, in=288, looseness = 10] node[auto,swap] {$\Lambda_2$} (1);
        \path[<-] (1) edge node[auto,swap] {$Z_2$} (2);
        \path[<-] (2) edge[out=288, in=0, looseness = 10] node[auto,swap] {$\Lambda_3$} (2);
        \path[<-] (2) edge node[auto,swap] {$Z_3$} (3);
        \path[<-] (3) edge[out=0, in=72, looseness = 10] node[auto,swap] {$\Lambda_4$} (3);
        \path[<-] (3) edge node[auto,swap] {$Z_4$} (4);
        \path[<-] (4) edge[out=72, in=144, looseness = 10] node[auto,swap] {$\Lambda_5$} (4);
        \path[<-] (4) edge node[auto,swap] {$Z_5$} (0);
    \end{tikzpicture}
\end{center}

The following notation will be convenient for studying such graph moments.
\begin{notation} \label{not: loop graph}
  Let $T$ be a \tdg{} with labeling $X: E(T) \to M_N(\bC)$.  Let $\Lambda: V(T) \to \Delta(M_N(\bC))$ be an assignment of diagonal matrices to each vertex of the \dg{}.  Then we denote by $\mathring{T}$ the \tdg{} obtained by adding a self-looping edge to each vertex $v \in V(T)$ and labeling this edge with the diagonal matrix $\Lambda_v$.
\end{notation}

Thus, the example \dg{} above could be written as $\mathring{T}$ where $T$ is a cycle \dg{} labeled by the matrices $Z_1$, \dots, $Z_5$.  Note that in the setup of Notation \ref{not: loop graph}, we have
\[
\tau_N(\mathring{T}) = \frac{1}{N^{\# \Comp(T)}} \sum_{i: V(T) \to [N]} \prod_{v \in V(T)} (\Lambda_v)_{i(v),i(v)} \prod_{e \in E(T)} (X_e)_{i(e_+),i(e_-)}.
\]

\subsection{Combinatorial computation of graph moments} \label{subsec: combinatorial graph moments}

The arguments in this section will introduce several \tdg{}s constructed from a given \tdg{}.
We invite the reader to follow Appendix~\ref{asubsec:tgraph ex} in parallel for a worked example, beginning with Notation~\ref{not: rose}.

Based on the discussion above, we wish to express the asymptotics of certain traffic moments.
It will be convenient to work with a slight generalization of the setting of Theorem~\ref{thm: permutation model}.
\begin{itemize}
  \item Let $T$ be a weakly connected \dg{} together with a coloring $\chi: E(T) \to \mathcal{C}$.
  \item Let $\set{\Sigma_c^{(N)} : c \in \cC}$ be as in the statement of Theorem~\ref{thm: permutation model} and for each $c \in \cC$ let $\sigma_c$ be the permutation (of $[N]^{\cS_c}$) corresponding to $\Sigma_c$.
  \item For each $e \in E$, let $X_e \in \bigotimes_{\cS_{\chi(e)}} M_N(\bC)$ be a matrix and set
  \[ \uu{X}_e = \paren{\Sigma^{(N)}_{\chi(e)}}^* X_e \Sigma^{(N)}_{\chi(e)} \otimes I_N^{\otimes\cS\setminus\cS_{\chi(e)}} \in \bigotimes_{\cS}M_N(\bC). \]
\end{itemize}

\begin{notation}
  For ease of notation, we will treat elements of $[N]^{\cS_{\chi(e)}}$ as indices to the matrix $X_e$, and elements of $[N]^{\cS}$ as indices to the matrix $\uu{X}_e.$
\end{notation}

We will view $T$ as a (random) \tdg{} with edge labelled by the $\uu{X}_e$'s, and applying the above definition we have
\begin{align} \label{eq:traffictrace}
\tau_{N^{\# \cS}}(T) &= \frac{1}{N^{\# \mathcal{S}}} \sum_{i: V(T) \to [N]^{\mathcal{S}}} \prod_{e \in E(T)} (\uu{X}_e)_{i(e_+),i(e_-)}, \\
\label{eq:traffictrace2}
	\tau_{N^{\# \cS}}^\circ(T) &= \frac{1}{N^{\# \mathcal{S}}} \sum_{i: V(T) \hookrightarrow [N]^{\mathcal{S}}} \prod_{e \in E(T)} (\uu{X}_e)_{i(e_+),i(e_-)}.
\end{align}
Similarly, with $\Lambda : V(T) \to \Delta(M_N(\bC))$ and $\mathring{T}$ as in Notation~\ref{not: loop graph}, we have
\begin{equation} \label{eq:traffictrace3}
\tau_{N^{\# \cS}}(\mathring{T}) = \frac{1}{N^{\# \mathcal{S}}} \sum_{i: V(T) \to [N]^{\mathcal{S}}} \prod_{v \in V(T)} (\Lambda_v)_{i(v),i(v)} \prod_{e \in E(T)} (\uu{X}_e)_{i(e_+),i(e_-)}.
\end{equation}

To make the computation of $\bE\tau(\mathring{T})$ (an exercise in counting permutations) tractable, we first separate terms based on which of their $i$-coordinates agree for which vertices.
For each vertex $v$, $i(v)$ is an element of $[N]^{\cS}$, or in other words a tuple indexed by $\cS$.
We let $i_s(v)$ be the $s$-coordinate of this tuple, so that $i_s: V(T) \to [N]$ for each $s$.
Let $\ker(i_s)$ be the partition of $V(T)$ where $v$ and $w$ are in the same block if and only if $i_s(v) = i_s(w)$.
We can then sort the terms in the sum \eqref{eq:traffictrace3} based on the values of $\ker(i_s)$ for $s \in \cS$:
\begin{equation} \label{eq:kernelexpansion}
  \tau_{N^{\# \cS}}(\mathring{T}) = \sum_{\pi \in \mathcal{P}(V(T))^{\cS}} \frac{1}{N^{\# \mathcal{S}}} \sum_{\substack{i: V(T) \to [N]^{\mathcal{S}} \\ \forall s \in \mathcal{S}, \ker(i_s) = \pi_s }} \prod_{v \in V(T)} (\Lambda_v)_{i(v),i(v)} \prod_{e \in E(T)} (\uu{X}_e)_{i(e_+),i(e_-)},
\end{equation}
where $\pi_s \in \mathcal{P}(V(T))$ is the $s$-th coordinate of $\pi$.
We write
\begin{equation} \label{eq:gamma}
  \gamma_N(T,\pi) = \frac{1}{N^{\# \mathcal{S}}}
  \sum_{\substack{i: V(T) \to [N]^{\mathcal{S}} \\ \forall s \in \mathcal{S}, \ker(i_s) = \pi_s }}
  \prod_{v \in V(T)} (\Lambda_v)_{i(v),i(v)}
  \prod_{e \in E(T)} (\uu{X}_e)_{i(e_+),i(e_-)}.
\end{equation}

Specifically, the role of $\pi$ is to record which matrix indices agree with each other: $v \sim_{\pi_s} w$ means that the $s$-component of the index at vertex $v$ is equal to that of the index at vertex $w$.

In order for $\gamma_N(T,\pi)$ to be non-zero, the partitions $\pi_s$ must satisfy certain conditions.

\begin{notation} \label{not: rose}
  Given a subset $C \subseteq \cC$, we will write $T|_C$ for the subgraph of $T$ containing only edges with colors in $C$.
  For $c \in \cC$, we will write $T|_c$ as shorthand for $T|_{\set{c}}$.

  Let $\rho_s$ be the partition of $V(T)$ whose blocks are the vertex sets of the weakly connected components of $T|_{\cC\setminus \cC_s}$.
  In other words, $v \sim_{\rho_s} w$ if and only if $v$ and $w$ are connected by a (direction agnostic) path of edges colored by colors not in $\cC_s$.
\end{notation}

\begin{lemma} \label{lem:vanishingconditions}
  We have $\gamma_N(T,\pi) = 0$ unless $\pi_s \geq \rho_s$ for all $s \in \cS$.
\end{lemma}
\begin{proof}
Let $s \in \cS$ and $e \in E(T)$ and suppose that $\chi(e) \not \in \cC_s$.
Recall that
\[
  \uu{X}_e = \Sigma_{\chi(e)}^* X_e \Sigma_{\chi(e)} \otimes I_N^{\otimes\cS\setminus\cS_{\chi(e)}};
\]
in particular, $(\uu{X}_e)_{i,j}$ will be zero if $i_s \neq j_s$.

It follows that if $i : V(T) \to [N]^{\cS}$ with $\ker(i_s) = \pi_s$ for each $s$ but $\gamma_N(T,\pi) \neq 0$, then $i(e_+)_s = i(e_-)_s$ whenever $\chi(e) \notin \cC_s$, and in particular $i(v) = i(w)$ whenever $v$ and $w$ are connected by edges colored with colors not in $\cC_s$, i.e., whenever $v \sim_{\rho_s} w$.

Hence, $\pi_s \geq \rho_s$ as desired.
\end{proof}

We now wish to evaluate the remaining $\gamma_N(T,\pi)$ terms combinatorially, which we will accomplish in Proposition~\ref{prop:trafficevaluation}.
First, the number of terms in the summation to compute $\gamma_N(T,\pi)$ can be described as follows using basic combinatorics.

\begin{lemma} \label{lem:counting}
\[
\# \{i: V(T) \to [N]^{\cS}: \forall s, \ker(i_s) = \pi_s \}
= \prod_{s \in \cS} \frac{N!}{(N-\# \pi_s)!}
\leq N^{\sum_{s\in\cS}\#\pi_s}.
\]
\end{lemma}

In order to state Proposition \ref{prop:trafficevaluation}, we introduce a few more pieces of notation.

\begin{notation}\label{not:lambdaN}
Let
\[
\lambda_N(T,\pi) = \frac{1}{N^{\sum_{s \in \cS} \# \pi_s}} \sum_{\substack{i: V(T) \to [N]^{\mathcal{S}} \\ \forall s \in \mathcal{S}, \ker(i_s) = \pi_s }}
\prod_{v \in V(T)} (\Lambda_v)_{i(v),i(v)}.
\]
Note that if $\norm{\Lambda_v} \leq R$ for all $v$, then $|\lambda_N(T,\pi)| \leq R^{\# V}$ because of Lemma \ref{lem:counting}.
\end{notation}

\begin{notation} \label{not:omegapic}
  With $\pi \in \cP(V(T))^{\cS}$ and $c \in \cC$, we let $\omegapic$ be the partition
  \[
    \omegapic = \bigwedge_{s \in \cS_c} \pi_s.
  \]
  Let us denote $\Tpic = (T|_c)/\omegapic$, which is a (random) \tdg{}.
\end{notation}

Notice that although $\Tpic$ is random, its (injective) trace is deterministic as the only remaining edges are labelled by matrices conjugated by the same random permutation.

Moreover, since the non-scalar entries of $\uu{X}_e$ occur on the strings in $\cS_c$, if edges $e, f$ have common color then $e_+ \sim_{\pi_c} f_+$ precisely records the condition that the same (relevant portion of the) index will be used for both $e$ and $f$.

\begin{prop} \label{prop:trafficevaluation}
Assume $T$ is connected.  Let $\pi \in \mathcal{P}(V(T))^{\cS}$ satisfy $\pi_s \geq \rho_s$ for all $s \in \cS$.  Then
\begin{align*}
  \bE\gamma_N(T,\pi) &= N^{\sum_{s \in \cS} (\# \pi_s - 1)} \lambda_N(T,\pi) \prod_{c \in \mathcal{C}} \left( \frac{(N^{\# \cS_c} - \# V(\Tpic))!}{(N^{\# \cS_c})!} N^{\# \cS_c \# \Comp(\Tpic)} \tau_{N^{\#\cS_c}}^\circ(\Tpic) \right) \\
&= \left(1 + O(N^{-1})\right) N^{\sum_{s \in \cS} (\# \pi_s - 1) - \sum_{c \in \cC} \# \cS_c \# V(\Tpic)} \lambda_N(T,\pi) \prod_{c \in \cC} N^{\# \cS_c \# \Comp(\Tpic)} \tau_{N^{\# \cS_c}}^\circ(\Tpic)
\end{align*}
\end{prop}

\begin{proof}
Recall that $\uu{X}_e = (\Sigma_{\chi(e)}^t X_e \Sigma_{\chi(e)}) \otimes I_N^{\otimes\cS \setminus \cS_c}$.
Fix $i: V(T) \to [N]^{\cS}$ such that $\ker(i_s) = \pi_s \geq \rho_s$.
Let us express the $i(e_+)$, $i(e_-)$ entry of $\uu{X}_e$ through matrix multiplication.
Fixing $e$ and letting $c = \chi(e)$, we have
\begin{align*}
  (\uu{X}_e)_{i(e_+),i(e_-)} &= [(\Sigma_c^* X_e \Sigma_c) \otimes I_N^{\otimes\cS \setminus \cS_c}]_{i(e_+), i(e_-)} \\
	&= (\Sigma_c^* X_e \Sigma_c)_{i(e_+)|_{\cS_c}, i(e_-)|_{\cS_c}}
\end{align*}
since by assumption $i(e_+)_s = i(e_-)_s$ for $s \in \cS \setminus \cS_c$.
Then
\begin{align}
  (\Sigma_c^t X_e \Sigma_c)_{i(e_+)|_{\cS_c}, i(e_-)|_{\cS_c}} &= \sum_{j,k \in [N]^{\cS_c}} (\Sigma_c)_{j,i(e_+)} (X_e)_{j,k} (\Sigma_c)_{k,i(e_-)} \nonumber\\ &= \sum_{j,k \in [N]^{\cS_c}} \mathbf{1}_{\sigma_c(i(e_+)|_{\cS_c}) = j} \mathbf{1}_{\sigma_c(i(e_-)|_{\cS_c}) = k} (X_e)_{j,k},
\label{eq:summations}
\end{align}
where $\sigma_c$ is the permutation whose matrix is $\Sigma_c$.
To compute $\gamma_N(T,\pi)$ from \eqref{eq:gamma}, we must take the product of this expression over all $e \in E(T)$; expanding the product would correspond to a choice of $j, k \in [N]^{\cS_c}$ for each edge $e$.
Overall we sum over functions $j_c, k_c: \chi^{-1}(c) \to [N]^{\cS_c}$ for each $c \in \mathcal{C}$.

If the indicator functions in \eqref{eq:summations} do not vanish, we conclude that for each $e \in E(T)$ with $c = \chi(e)$, one has $\sigma_c(i(e_+)|_{\cS_c}) = j_c(e)$ and $\sigma_c(i(e_-)|_{\cS_c}) = k_c(e)$ for each edge $e$.
This, in turn, gives us the following equivalences:
for any edges $e, f \in E(T)$ with a common color $c$,
\begin{itemize}
  \item $e_+\sim_{\omegapic}f_+$ if and only if $j_c(e) = j_c(f)$ (as both equal $\sigma_c(i(e_+)|_{\cS_c}) = \sigma_c(i(f_+)|_{\cS_c})$);
  \item $e_+\sim_{\omegapic}f_-$ if and only if $j_c(e) = k_c(f)$ (as both equal $\sigma_c(i(e_+)|_{\cS_c}) = \sigma_c(i(f_-)|_{\cS_c})$); and
  \item $e_-\sim_{\omegapic}f_-$ if and only if $k_c(e) = k_c(f)$ (as both equal $\sigma_c(i(e_-)|_{\cS_c}) = \sigma_c(i(f_-)|_{\cS_c})$).
\end{itemize}

The foregoing implies that we only need to sum over functions $j_c$ and $k_c$ that can be expressed in terms of some injective function $\ell_c: V(\Tpic) \to [N]^{\cS_c}$ through the relations
\begin{align*}
	j_c(e) &= \ell_c([e_+]_{\pi_c}) & k_c(e) &= \ell_c([e_-]_{\pi_c}).
\end{align*}

We thus obtain
\begin{multline*}
  \prod_{e \in E(T)} (\uu{X}_e)_{i(e_+),i(e_-)} \\
	= \prod_{c \in \mathcal{C}} \left( \sum_{\ell_c: V(\Tpic) \hookrightarrow [N]^{\cS_c}}  \prod_{e \in E(\Tpic)}  \mathbf{1}_{\sigma_c(i(e_+)|_{\cS_c}) = \ell_c([e_+]_{\pi_c})} \mathbf{1}_{\sigma_c(i(e_-)|_{\cS_c})=\ell_c([e_-]_{\pi_c})} (X_e)_{\ell_c([e_+]_{\pi_c}),\ell_c([e_-]_{\pi_c})} \right). \end{multline*}
Now let us compute the expectation over the random choice of permutations.  The only terms that depend on the permutation are the indicator functions, and
\begin{multline*}
\mathbb{E} \left[ \prod_{c \in \cC} \prod_{e \in E(\Tpic)}  \mathbf{1}_{\sigma_c(i(e_+)|_{\cS_c}) = \ell_c([e_+]_{\pi_c})} \mathbf{1}_{\sigma_c(i(e_-)|_{\cS_c})=\ell_c([e_-]_{\pi_c})} \right] \\
= \mathbb{P} \left( \forall c \in \cC, \forall e \in E(\Tpic), \sigma_c(i(e_+)|_{\cS_c}) = \ell_c([e_+]_{\pi_c}) \text{ and }\sigma_c(i(e_-)|_{\cS_c}) = \ell_c([e_-]_{\pi_c}) \right).
\end{multline*}
Because the permutations for different colors are independent, this equals
\[
\prod_{c \in \mathcal{C}} \mathbb{P} \left( \forall e \in E(\Tpic), \sigma_c(i(e_+)|_{\cS_c}) \\ = \ell_c([e_+]_{\pi_c}) \text{ and }\sigma_c(i(e_-)|_{\cS_c}) = \ell_c([e_-]_{\pi_c}) \right).
\]
Fix $c \in \cC$, and let us compute the probability for this $c$.
Since $\ell_c$ is injective, it takes $\# V(\Tpic)$ distinct values in $[N]^{\cS_c}$, and correspondingly the entries of $i|_{\cS_c}$ relevant to the computation for $c$ take $\# V(\Tpic)$ distinct values.
Hence, we need to compute the probability that a uniformly random permutation of $N^{\# \cS_c}$ elements sends a prescribed list of $\# V(\Tpic)$ of those elements to another prescribed list of $\# V(\Tpic)$ in order.
This evaluates to
\[
\mathbb{P} \left( \forall e \in E(\Tpic), \sigma_c(i(e_+)|_{\cS_c}) = \ell_c([e_+]_{\pi_c}) \text{ and }\sigma_c(i(e_-)|_{\cS_c}) = \ell_c([e_-]_{\pi_c}) \right) = \frac{(N^{\# \cS_c} - \# V(\Tpic))!}{(N^{\# \cS_c})!}.
\]
Therefore,
\begin{align*}
  \mathbb{E} \prod_{e \in E(T)} (\uu{X}_e)_{i(e_+),i(e_-)}
	&= \prod_{c \in \mathcal{C}} \left( \frac{(N^{\# \cS_c} - \# V(\Tpic))!}{(N^{\# \cS_c})!} \sum_{\ell_c: V(\Tpic) \hookrightarrow [N]^{\cS_c}} \prod_{e \in E(\Tpic)} (X_e)_{\ell_c([e_+]_{\pi_c}), \ell_c([e_-]_{\pi_c})} \right) \\
	&= \prod_{c \in \mathcal{C}} \left( \frac{(N^{\# \cS_c} - \# V(\Tpic))!}{(N^{\# \cS_c})!} N^{\# \cS_c \# \Comp(\Tpic)} \tau_{N^{\# \cS_c}}^\circ(\Tpic) \right).
\end{align*}
Hence
\begin{align*}
  \mathbb{E} \gamma_N(T,\pi) &= \frac{1}{N^{\# \mathcal{S}}}
  \sum_{\substack{i: V(T) \to [N]^{\mathcal{S}} \\ \forall s \in \mathcal{S}, \ker(i_s) = \pi_s }}
  \prod_{v \in V(T)} (\Lambda_v)_{i(v),i(v)}
  \mathbb{E} \left[ \prod_{e \in E(T)} (\uu{X}_e)_{i(e_+),i(e_-)} \right] \\
  &= \frac{1}{N^{\# \mathcal{S}}}
  \sum_{\substack{i: V(T) \to [N]^{\mathcal{S}} \\ \forall s \in \mathcal{S}, \ker(i_s) = \pi_s }}
  \prod_{v \in V(T)} (\Lambda_v)_{i(v),i(v)} \prod_{c \in \mathcal{C}} \left( \frac{(N^{\# \cS_c} - \# V(\Tpic))!}{(N^{\# \cS_c})!} N^{\# \cS_c \# \Comp(\Tpic)} \tau_{N^{\# \cS_c}}^\circ(\Tpic) \right) \\
  &= \frac{N^{\sum_{s \in \cS} \# \pi_s}}{N^{\# \cS}} \lambda_N(T,\pi) \prod_{c \in \mathcal{C}} \left( \frac{(N^{\# \cS_c} - \# V(\Tpic))!}{(N^{\# \cS_c})!} N^{\# \cS_c \# \Comp(\Tpic)} \tau_{N^{\# \cS_c}}^\circ(\Tpic) \right).
\end{align*}
After observing that $\sum_{s \in \cS} \# \pi_s - \# \cS = \sum_{s \in \cS} (\# \pi_s - 1)$, we obtain the first formula asserted in the proposition.  The second follows using the fact that
\[
  \frac{(N^{\# \cS_c} - \# V(\Tpic))!}{(N^{\# \cS_c})!} = (1 + O(1/N)) N^{-\# \cS_c \# V(\Tpic)}
\]
and then combining the terms in the exponent.
\end{proof}

By combining \eqref{eq:kernelexpansion}, Lemma \ref{lem:vanishingconditions}, and Proposition \ref{prop:trafficevaluation}, we obtain the following result.

\begin{cor} \label{cor:trafficevaluation}
\begin{align*}
  \mathbb{E} \tau_{N^{\# \cS}}(\mathring{T}) = \sum_{\substack{\pi \in \mathcal{P}(V(T))^{\cS} \\ \forall s \in \cS, \pi_s \geq \rho_s}}& N^{\sum_{s \in \cS} (\# \pi_s - 1)} \lambda_N(T,\pi)\\
  &\prod_{c \in \mathcal{C}} \left( \frac{(N^{\# \cS_c} - \# V(\Tpic))!}{(N^{\# \cS_c})!} N^{\# \cS_c \# \Comp(\Tpic)} \tau_N^\circ(\Tpic) \right) \\
  = \sum_{\substack{\pi \in \mathcal{P}(V(T))^{\cS} \\ \forall s \in \cS, \pi_s \geq \rho_s}}& \left(1 + O(N^{-1})\right) N^{\sum_{s \in \cS} (\# \pi_s - 1) } \lambda_N(T,\pi)\\
    &\prod_{c \in \cC} N^{\# \cS_c \paren{\# \Comp(\Tpic) -   \# V(\Tpic)}}\tau_{N^{\# \cS_c}}^\circ(\Tpic) \\
\end{align*}
\end{cor}

\subsection{Growth rates and graph theory}\label{subsec: growth graphs}

In this section, we continue with the same setup as \S \ref{subsec: combinatorial graph moments}.
In particular, throughout the entire section we only concern ourselves with partitions $\pi$ satisfying $\pi_s \geq \rho_s$ for each $s \in \cS$, as these are the only ones with corresponding terms in the sum in Corollary~\ref{cor:trafficevaluation}.
In order to determine which of these terms contribute in the limit as $N \to \infty$, we will apply an estimate of Mingo and Speicher \cite{MS2012} that estimates traffic moments in terms of the operator norms of the corresponding matrices.  The power of $N$ appearing in the estimate is defined in terms of two-edge-connected components.

\begin{itemize}
  \item A \emph{cut edge} in a (di)graph $G$ is an edge $e$ such that $G \setminus e$ has one more (weakly) connected component than $G$.
  \item A (di)graph is \emph{(weakly) two-edge-connected} if it is (weakly) connected and has no cut-edges.
  \item The \emph{(weakly) two-edge connected components} of $G$ are the maximal two-edge-connected sub-(di)graphs of $G$.
  \item For a (di)graph $G$, let $\mathcal{F}(G)$ be the undirected graph with vertices given by the (weakly) two-edge-connected components of $G$, with an edge between two two-edge-connected components if and only if there is a cut edge in $G$ between those two components.
    Note that $\mathcal{F}(G)$ is a forest, with one connected component corresponding to each connected component of $G$.
  \item For a (di)graph $G$, $\mathfrak{f}(G)$ denotes the total number of leaves in $\mathcal{F}(G)$, where an isolated vertex in $\mathcal{F}(G)$ counts as two leaves.
\end{itemize}

\begin{remark}
It can be checked that the two-edge-connected components of a (di)graph $G$ are precisely the connected components of the graph obtained by removing all the cut edges from $G$.
In particular, two vertices $v$ and $w$ are in the same two-edge-connected component if and only if they are joined by a path in $G$ containing no cut edges.
\end{remark}

\begin{thm}[{\cite[Theorem 6]{MS2012}}]
  For a \tdg{} $T$ with labels $A_e \in M_N(\bC)$, we have
  \[
    N^{\# \Comp(T)} |\tau_N(T)| \leq C_T N^{\mathfrak{f}(T)/2} \prod_{e \in E} \norm{A_e},
  \]
  where $C_T$ is a constant that only depends on the underlying graph for $T$.
\end{thm}

One can give a similar bound for the injective trace.  This follows from the M\"obius inversion expressing $\tau_N^0$ in terms of $\tau_N$ as noted in \cite[Lemma 4.14 (4.9 of the arXiv version)]{ACDGM2021}.

\begin{cor}
\label{cor: mingo speicher injective}
  For a \tdg{} $T$ with labels $A_e \in M_N(\bC)$, we have
\[
N^{\# \Comp(T)} |\tau_N^\circ(T)| \leq C_T' N^{\mathfrak{f}(T)/2} \prod_{e \in E} \norm{A_e},
\]
where $C_T'$ is a constant that only depends on the underlying graph for $T$.
\end{cor}

We can apply this theorem together with Proposition \ref{prop:trafficevaluation} to conclude the following.

\begin{cor} \label{cor: traffic growth bound}
Suppose that $\norm{\Lambda_v^{(N)}} \leq R$ and $\norm{X_e^{(N)}} \leq R$.  Then
\[
  |\mathbb{E} \gamma_N(T,\pi)| \leq C_T' R^{\# V(T) + \# E(T)} (1 + O(N^{-1})) N^{\sum_{s \in \cS} (\# \pi_s - 1) + \sum_{c \in \cC} \# \cS_c (\mathfrak{f}(\Tpic)/2 - \# V(\Tpic) )},
\]
where $C_T'$ is a constant that only depends on the underlying graph for $T$.
\end{cor}

\begin{proof}
We already noted before that $|\lambda_N(T,\pi)| \leq R^{\# V(T)}$.  By Corollary \ref{cor: mingo speicher injective},
\begin{equation}
\label{eq:longday}
N^{\# \cS_c \# \Comp(\Tpic)} |\tau_{N^{\# \cS}}^\circ(\Tpic)| \leq C_T' R^{\# E(\Tpic)} N^{\#\cS_c\mathfrak{f}(\Tpic)/2}.
\end{equation}
We plug these estimates into Proposition \ref{prop:trafficevaluation} and combine the exponents.
\end{proof}

Next, our goal is to figure out what power of $N$ occurs in front of each term in the sum in Corollary \ref{cor:trafficevaluation}.  In particular, we want to show that
\[
\sum_{s \in \cS} (\# \pi_s - 1) + \sum_{c \in \cC} \# \cS_c (\mathfrak{f}(\Tpic)/2 - \# V(\Tpic) ) \leq 0,
\]
and characterize when equality is achieved.
For this argument, we need a few more auxiliary \dg{}s.

\begin{notation}\label{not:gpis}
  Let $\gps = (T|_{\cC_s})/\pi_s$ be the \dg{} obtained from $T$ by identifying the vertices according to $\pi_s$ (in particular, vertices are identified if they are connected by a path of non-$\cC_s$-colored edges since $\pi_s \geq \rho_s$), deleting all edges that are not $\cC_s$-colored.
  (Notice that all non-$\cC_s$-colored edges would be contracted to loops in $T/\pi_s$.)

  This is similar to the notation $\Tpic = (T|_c)/\pi_c$, except we are now dealing with a collection of colors present on one string rather than a single color.
\end{notation}

\begin{notation}\label{not:hsc}
  For $c \in \cC_s$, $\omegapic$ refines $\pi_s$ and therefore there is an induced \dg{} homomorphism $\Tpic \to T/\pi_s$.
  Since all the edges in $\Tpic$ are $c$-colored, the image of this map is in fact in the sub-\dg{} $\gps \subseteq T/\pi_s$.
  We write $\hsc : \Tpic \to \gps$ for this homomorphism, so that for $v \in V(T)$, $\hsc([v]_{\pi_c}) = [v]_{\pi_s}$.
\end{notation}

\begin{defn}\label{def:gcc}
Now define the \emph{graph of colored components} $\GCC(T,\pi,s)$ as follows.
$\GCC(T,\pi,s)$ will be a bipartite (multi)graph where the two vertex sets are $V(\gps)$ and $\bigsqcup_{c \in \cC_s} \Comp(\Tpic)$.
The set of edges in $\GCC(T,\pi,s)$ is defined to be
\[ \bigsqcup_{c \in \cC_s} V(\Tpic), \]
where $v \in V(\Tpic)$ is treated as an edge between the connected component of $\Tpic$ containing $v$ and $\hsc(v)$.
(Hence, multiple edges are allowed between the same pair of vertices.)
\end{defn}

The following lemma shows that for $s \in \cS$ and $c \in \cC_s$, walks in $\Tompic$ give rise to walks in $\GCC(T,\pi,s)$; this will be useful to us in the proofs that follow.
In particular, it shows that $\GCC(T,\pi,s)$ is connected or two-edge connected if $T$ is; it also allows us to build paths in $\GCC(T,\pi,s)$ which start or end at particular vertices and visit (or avoid) particular components.
The lemma holds under weaker but more technical hypotheses, but the version as stated suffices for our purposes.

\begin{lemma}
  \label{lem:inducedgccpath}
  Fix $s \in \cS$, let $(e_j)_{j=1}^m \in E(T)$ be a sequence of edges, and let
  \[\set{ j \in [m] \colon \chi(e_j) \in \cC_s } = \set{j_1 < j_2 < \ldots < j_n}.\]
  For each $i$, let $A_i$ be the connected component of $T_{\pi,\chi(e_i)}$ which contains $e_i$.

  Suppose that for some $c \in \cC_s$, $(e_j)_j$ forms a walk when viewed in $E(\Tompic)$.
  Then there is a walk in $\GCC(T,\pi,s)$ from $[(e_1)_-]_{\pi_s}$ to $[(e_m)_+]_{\pi_s}$ along the edges
  \[
    [(e_{j_1})_-]_{\pi_{\chi(e_{j_1})}},\
    [(e_{j_1})_+]_{\pi_{\chi(e_{j_1})}},\
    [(e_{j_2})_-]_{\pi_{\chi(e_{j_2})}},\
    [(e_{j_2})_+]_{\pi_{\chi(e_{j_2})}},\
    \ldots,\
    [(e_{j_n})_+]_{\pi_{\chi(e_{j_n})}},
  \]
  which visits precisely the vertices
  \begin{align}
  \label{eq: gcc path}
    [(e_1)_-]_{\pi_s} = [(e_{j_1})_-]_{\pi_s},\
    A_{j_1},\
    [(e_{j_1})_+]_{\pi_s} = [(e_{j_2})_-]_{\pi_s},\
    A_{j_2},\
    \ldots,\
    A_{j_n},\
    [(e_{j_n})_+]_{\pi_s} = [(e_{j_m})_+]_{\pi_s}.
\end{align}
\end{lemma}

\begin{proof}
  Notice that $[(e_{j_i})_\pm]_{\pi_{\chi(e_j)}} \in V(T_{\pi,\chi(e_j)}) \subseteq E(\GCC(T,\pi,s))$ is incident with both $[(e_{j_i})_\pm]_{\pi_s}$ (because $\hsc([(e_{j_i})_\pm]_{\pi_{\chi(e_j)}}) = [(e_{j_i})_\pm]_{\pi_s}$) and $A_{j_i}$ (as $A_{j_i}$ contains $e_{j_i}$ and therefore $[(e_{j_i})_\pm]_{\pi_{\chi(e_j)}}$).

  The only thing to check is that the equalities claimed in \eqref{eq: gcc path} hold.
  Take $c \in \cC_s$ so that $(e_j)_j$ forms a walk in $\Tompic$.
  Then in particular, $[(e_{j_i})_+]_{\omegapic}$ is connected to $[(e_{j_{i+1}})_-]_{\omegapic}$ by edges $e_{j_i+1}, \ldots, e_{j_{i+1}-1}$ which have colors not in $\cC_s$.
  It follows that $(e_{j_i})_+ \sim_{\rho_s} (e_{j_{i+1}})_-$ and (as $\rho_s \leq \pi_s$) $(e_{j_i})_+ \sim_{\pi_s} (e_{j_{i+1}})_-$ as well.
  We have $(e_1)_- \sim_{\pi_s} (e_{j_1})_-$ and $(e_{j_n})_+ \sim_{\pi_s} (e_m)_+$ by an identical argument.
\end{proof}

With this in hand, we will prove the following generalization of \cite[Lemma 4.12]{ACDGM2021}.
From this point onward, we will be specializing to the case that the test graph $T$ is two-edge-connected.
This holds, in particular, when $T$ corresponds to the trace of a product, as indicated by the discussion in \S~\ref{subsec:trafficmoments}; we will see in \S~\ref{subsec: asym graph ind over diag} that the slightly more complicated terms we will need to prove Theorem~\ref{thm: permutation model} also correspond to two-edge-connected graphs.

\begin{lemma} \label{lem:leafcount}
Assume $T$ is two-edge connected.  With the setup above,
\begin{equation} \label{eq:componentcount}
\sum_{s \in \cS} (\# \pi_s - 1) + \sum_{c \in \cC} \# \cS_c (\mathfrak{f}(\Tpic)/2 - \# V(\Tpic) ) \leq 0,
\end{equation}
Moreover, equality is achieved if and only if $\GCC(T,\pi,s)$ is a tree for all $s \in \cS$, and in this case $\mathfrak{f}(\Tpic) / 2 = \#\Comp(\Tpic)$.
\end{lemma}

\begin{proof}
The term we want to bound can be written as
\begin{equation*}
\sum_{s \in \cS} (\# \pi_s - 1) + \sum_{c \in \cC} \sum_{s \in \cS} \delta_{s \in \cS_c} (\mathfrak{f}(\Tpic)/2 - \# V(\Tpic))
=
\sum_{s \in \cS} (\# \pi_s - 1) + \sum_{s\in\cS}\sum_{c \in \cC_s} (\mathfrak{f}(\Tpic)/2 - \# V(\Tpic)).
\end{equation*}
Hence, it suffices to show that for every $s$, we have
\begin{equation} \label{eq:combinatorialestimate2}
\# \pi_s - 1 + \sum_{c \in \cC_s} (\mathfrak{f}(\Tpic)/2 - \# V(\Tpic)) \leq 0,
\end{equation}
and if this inequality is true for all $s$, then equality will hold in \eqref{eq:componentcount} if and only if equality holds in \eqref{eq:combinatorialestimate2} for all $s$.
Note that
\[
\# \pi_s = \# V(\gps) = \# V(\GCC(T,\pi,s)) - \sum_{c \in \cC_s} \# \Comp(\Tpic).
\]
By definition, the edges of $\GCC(T,\pi,s)$ correspond to $\bigsqcup_{c \in \cC_s} V(\Tpic)$, so that
\[
\sum_{c \in \cC_s} \# V(\Tpic) = \# E(\GCC(T,\pi,s)).
\]
Hence, the inequality we want is equivalent to
\[
\# V(\GCC(T,\pi,s)) - 1 + \sum_{c \in \cC_s} (\mathfrak{f}(\Tpic)/2 - \# \Comp(\Tpic)) \leq \# E(\GCC(T,\pi,s)).
\]
We can furthermore write
\[
\mathfrak{f}(\Tpic) / 2 - \# \Comp(\Tpic) = \sum_{A \in \Comp(\Tpic)} (\mathfrak{f}(A)/2 - 1).
\]
Thus, the inequality we want is
\begin{equation} \label{eq: nice graph inequality}
\sum_{c \in \cC_s} \sum_{A \in \Comp(\Tpic)} (\mathfrak{f}(A)/2 - 1) \leq \# E(\GCC(T,\pi,s)) - \# V(\GCC(T,\pi,s) + 1.
\end{equation}
Note that $\GCC(T,\pi,s)$ is connected (as $T$ is) and hence the right-hand side is nonnegative.

For $c \in \cC_s$ and $A \in \Comp(\Tpic)$, we claim that if $A$ contains a cut edge, then: (1) $A$ is part of a cycle in $\GCC(T,\pi,s)$; and (2) each leaf in $\cF(A)$ contains a vertex which, when viewed as an edge in $\GCC(T,\pi,s)$, is part of a cycle.
Suppose that $\mathcal{F}(A)$ is not a single vertex; then there is a cut edge $e$ in $A$ which divides $A$ into pieces $B_1$ and $B_2$.
Because $T$ is two-edge-connected, so is its quotient $\Tompic$.
Thus, there is some path in $\Tompic$ from a vertex of $B_1$ to a vertex in $B_2$ that does not use the edge $e$.
By considering the segment of the path from the last time it exits $B_1$ to the first time afterward that it enters $B_2$, we can assume without loss of generality that the path uses only edges not in $A$.
Let $v\in B_1$ and $w \in B_2$ be the endpoints of this path; then Lemma~\ref{lem:inducedgccpath} tells us that there is a walk (and therefore a path) in $\GCC(T,\pi,s)$ from $\hsc(v)$ to $\hsc(w)$ which does not visit $A \in V(\GCC(T, \pi, s))$.
However both $\hsc(v)$ and $\hsc(w)$ are adjacent to $A$ (using $v$ and $w$, respectively, viewed as (distinct!) edges in $\GCC(T,\pi,s)$), and so we may complete the given path to a cycle by passing through $A$, establishing (1).
If we started with a cut edge incident with a leaf in $\cF(A)$, then $B_1$ or $B_2$ will be this leaf, and $v$ or $w$ will be a vertex establishing (2).


 Part (1) of the above claim implies that if $\GCC(T,\pi,s)$ is a tree, then for each $c \in \cC_s$ and each $A \in \Comp(\Tpic)$, $\mathcal{F}(A)$ is a single point, and hence $\mathfrak{f}(A) = 2$ by the definition of $\mathfrak{f}$.
 Therefore, if $\GCC(T,\pi,s)$ is a tree, then both sides of \eqref{eq: nice graph inequality} are zero, hence equality is achieved.

Now suppose that $\GCC(T,\pi,s)$ is not a tree.
Consider ``pruning'' $\GCC(T,\pi,s)$ by removing all vertices of degree one and their corresponding edges iteratively until there are no degree-one vertices left, and call this graph $\mathcal{G}$.
$\mathcal{G}$ is independent of the order of pruning operations, as it contains precisely those vertices which are contained in at least one cycle; it is not a tree, and every vertex in $\mathcal{G}$ has degree at least two.
We have
 \[
\# E(\mathcal{G}) - \# V(\mathcal{G}) = \# E(\GCC(T,\pi,s)) - \# V(\GCC(T,\pi,s)),
 \]
 as this quantity is maintained at every step of the pruning operation.

Part (1) of the above claim showed that every colored component $A$ in $\GCC(T,\pi,s)$ such that $\mathcal{F}(A)$ is not a point must be part of a nontrivial cycle in $\GCC(T,\pi,s)$, and hence is one of the vertices that remains in $\mathcal{G}$ after the pruning process.
Moreover, by part (2) it follows that $\deg_{\mathcal{G}}(A) \geq \mathfrak{f}(A)$ (note that the edges guaranteed by (2) are necessarily distinct because the vertices come from distinct leaves).
The vertices of $\mathcal{G}$ include these colored components, as well as some colored components where $\mathcal{F}(A)$ is a point and some vertices from $V(\gps)$.
For every vertex of $\mathcal{G}$ that is a colored component $A$ where $\mathcal{F}(A)$ is a singleton, we also have $\mathfrak{f}(A) = 2 \leq \deg_{\mathcal{G}}(A)$ since every vertex $x$ in $\mathcal{G}$ must have degree at least $2$ by construction.  Similarly, for the vertices in $\mathcal{G}$ that come from $V(\gps)$, we have $\deg_{\mathcal{G}}(x)/2 - 1 \geq 0$.  Putting all this information together,
\begin{align*}
\sum_{c \in \cC_s} \sum_{A \in \Comp(\Tpic)} (\mathfrak{f}(A)/2 - 1) &\leq \sum_{x \in V(\mathcal{G})} (\deg_{\mathcal{G}}(x)/2 - 1) \\
&= \# E(\mathcal{G}) - \# V(\mathcal{G}) \\
&< \# E(\GCC(T,\pi,s)) - \# V(\GCC(T,\pi,s)) + 1.
\end{align*}
Therefore, \eqref{eq: nice graph inequality} holds with strict inequality.

We showed that when $\GCC(T, \pi, s)$ is a tree each connected component of $\Tpic$ is two-edge connected; thus the final equality claimed in the lemma holds as well.
\end{proof}

Combining Lemma \ref{lem:leafcount} with Corollary \ref{cor:trafficevaluation} and Corollary \ref{cor: traffic growth bound} (specifically, \eqref{eq:longday}), we obtain the following

\begin{cor} \label{cor:asymptoticexpansion2}
Let $T$ be two-edge connected and continue all the notation from above.  Then
\[
  \mathbb{E} \tau_{N^{\# \cS}}(\mathring{T}) = \sum_{\substack{\pi \in \mathcal{P}(V(T))^{\cS} \\ \forall s \in \cS, \pi_s \geq \rho_s \\ \forall s \in \cS, \GCC(T,\pi,s) \text{ is a tree}}} \lambda_N(T,\pi) \prod_{c \in \cC} \tau_{N^{\# \cS_c}}^\circ(\Tpic) + O(N^{-1}).
\]
Here the implicit constant in the $O(N^{-1})$ depends only on the graph $T$ and the operator norm bounds $\sup_{v} \norm{\Lambda_v^{(N)}}$ and $\sup_{e} \norm{X_e^{(N)}}$.
\end{cor}

The following observation will be useful for the next section:

\begin{lem} \label{lem: tree injective}
Fix $s$ and $c$.  Suppose that $\GCC(T,\pi,s)$ is a tree.  Then the map $\hsc: \Tpic \to \gps$ is injective on each component of $\Tpic$.
\end{lem}

\begin{proof}
Suppose that two vertices $w_1$ and $w_2$ in a component $A$ of $\Tpic$ both map to the same vertex $v$ in $\gps$.  Then $w_1$ and $w_2$ both produce edges between $A$ and $v$ in $\GCC(T,\pi,s)$.  As $\GCC(T,\pi,s)$ is a tree, it does not have parallel edges and so $w_1 = w_2$.
\end{proof}

\subsection{Convergence in $L^2$}
\label{subsec: asym graph ind over diag}

In this section, we will prove that the convergence in Theorem~\ref{thm: permutation model} holds in $L^2$, and so we will use all the notation from its setup.
This amounts to showing that
\[
  \lim_{N \to \infty} \mathbb{E} \norm{\Delta_{N^{\#\cS}}[(Y_1^{(N)} - \Delta_{N^{\#\cS}}[Y_1^{(N)}]) \dots (Y_k^{(N)} - \Delta_{N^{\#\cS}}[Y_k^{(N)}])]}_2^2 = 0.
\]
This also implies that the convergence holds in probability.

Observe that
\begin{multline*}
\norm{\Delta_{N^{\#\cS}}[(Y_1^{(N)} - \Delta_{N^{\#\cS}}[Y_1^{(N)}]) \dots (Y_k^{(N)} - \Delta_{N^{\#\cS}}[Y_k^{(N)}])]}_2^2 \\
= \tr_{N^{\# \cS}}\left[
\Delta_{N^{\#\cS}}\left[(Y_k^{(N)} - \Delta_{N^{\#\cS}}[Y_k^{(N)}])^* \dots (Y_1^{(N)} - \Delta_{N^{\#\cS}}[Y_1^{(N)}])^* \right]\cdot\right.\\
\left.\Delta_{N^{\#\cS}}\left[(Y_1^{(N)} - \Delta_{N^{\#\cS}}[Y_1^{(N)}]) \dots (Y_k^{(N)} - \Delta_{N^{\#\cS}}[Y_k^{(N)}])\right] \right].
\end{multline*}
Our goal is to express the right-hand side as a linear combination of graph moments, i.e.\ the traces of test graphs studied in \S \ref{subsec: combinatorial graph moments}.

Notice that if we expand the products in the right hand side, we have terms of the form
\[
    \tr_{N^{\#\cS}}\sq{\Delta_{N^{\#\cS}}\sq{ \cdots } \Delta_{N^{\#\cS}}\sq{ \cdots } },
\]
where the omitted pieces are products of $Y_i^{(N)}$ and diagonal matrices $\Delta_{N^{\#\cS}}(Y_i^{(N)})$, or their adjoints; to such a term we associate a set $I \subseteq \set{\pm1, \ldots, \pm k}$ recording the indices where a $\Delta_{N^{\#\cS}}(Y_i^{(N)})$ was chosen, with the negative indices corresponding to the adjointed copies.
We will encode these terms as graph moments, using the same principles as in \S\ref{subsec:trafficmoments}.

For each $i \in [k]$, we have
\[
  Y_i^{(N)} =
  \Lambda_{i,1}^{(N)} \uu{X}_{i,1}^{(N)} \dots \Lambda_{i,\ell(i)}^{(N)} \uu{X}_{i,\ell(i)}^{(N)}.
\]
Let $B_i$ be the \tdg{} pictured below with vertices $u_{i,1}, \ldots, u_{i, \ell(i)+1}$ and directed edges connecting the vertices ordered from $u_{i, \ell(i)+1}$ to $u_{i, 1}$ labelled by $X_{i,\ell(i)}, \ldots, X_{i, 1}$.
Likewise let $B_i'$ be the test graph with vertices $u_{i,1}', \ldots, u_{i,\ell(i)+1}'$ and directed edges connecting the vertices ordered from $u_{i,1}$ to $u_{i,\ell(i)+1}$ labelled by $X_{i,1}^*, \ldots, X_{i,\ell(i)}^*$.
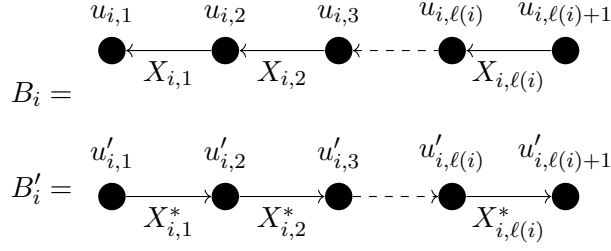
\begin{align*}
B_i &= \begin{tikzpicture}
    \node[circle,fill,label=$u_{i,1}$] (0) at (0,0) {};
    \node[circle,fill,label=$u_{i,2}$] (1) at (1.5,0) {};
    \node[circle,fill,label=$u_{i,3}$] (2) at (3,0) {};
    \node[circle,fill,label=$u_{i,\ell(i)}$] (3) at (4.5,0) {};
    \node[circle,fill,label=$u_{i,\ell(i)+1}$] (4) at (6,0) {};
    \draw[<-] (0) -- node[below] {$X_{i,1}$} (1);
    \draw[<-] (1) -- node[below] {$X_{i,2}$} (2);
    \draw[<-,dashed] (2) -- (3);
    \draw[<-] (3) -- node[below] {$X_{i,\ell(i)}$} (4);
\end{tikzpicture} \\
B_i' &= \begin{tikzpicture}[baseline]
    \node[circle,fill,label=$u_{i,1}'$] (0) at (0,0) {};
    \node[circle,fill,label=$u_{i,2}'$] (1) at (1.5,0) {};
    \node[circle,fill,label=$u_{i,3}'$] (2) at (3,0) {};
    \node[circle,fill,label=$u_{i,\ell(i)}'$] (3) at (4.5,0) {};
    \node[circle,fill,label=$u_{i,\ell(i)+1}'$] (4) at (6,0) {};
    \draw[->] (0) -- node[below] {$X_{i,1}^*$} (1);
    \draw[->] (1) -- node[below] {$X_{i,2}^*$} (2);
    \draw[->,dashed] (2) -- (3);
    \draw[->] (3) -- node[below] {$X_{i,\ell(i)}^*$} (4);
\end{tikzpicture}
\end{align*}

Construct a test graph by taking $B_{k}', \ldots, B_{1}', B_{1}, \ldots, B_{k}$ and joining them end to end to form two cycles: specifically, we identify $u_{i, 1}$ with $u_{i+1, \ell(i+1)+1}$ and $u_{1,1}$ with $u_{k, \ell(k)+1}$; and we identify $u_{i+1, \ell(i)+1}'$ with $u_{i, 1}'$ and $u_{1, \ell(1)+1}'$ with $u_{k, 1}'$.
Next, we identify $u_{1,1}$ with $u_{1,1}'$.
Call the resulting \tdg{} $T$, and assigning $\Lambda_{i,j}^{(N)}$ to $u_{i, j}$ and $(\Lambda_{i,j}^{(N)})^*$ to $u_{i,j}'$, we have a corresponding \tdg{} $\mathring{T}$ (here we multiply the two diagonal matrices assigned to $u_{1,1} = u_{1,1}'$).
The \tdg{} $\mathring{T}$ has the following form:
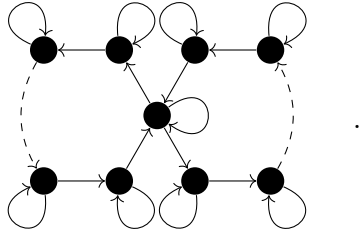
\begin{center}
    \begin{tikzpicture}[baseline,yshift=1ex]
        \node[circle,fill] (0) at (0,0) {};
        \node[circle,fill,shift={(1,0)}] (1) at (240:1) {};
        \node[circle,fill,shift={(1,0)}] (2) at (300:1) {};
        \node[circle,fill,shift={(1,0)}] (3) at (60:1) {};
        \node[circle,fill,shift={(1,0)}] (4) at (120:1) {};
        \node[circle,fill,shift={(-1,0)}] (1a) at (240:-1) {};
        \node[circle,fill,shift={(-1,0)}] (2a) at (300:-1) {};
        \node[circle,fill,shift={(-1,0)}] (4a) at (60:-1) {};
        \node[circle,fill,shift={(-1,0)}] (5a) at (120:-1) {};

        \path[->] (0) edge (1);
        \path[<-] (1) edge[out=204, in=276, looseness = 10] (1);
        \path[->] (1) edge (2);
        \path[<-] (2) edge[out=264, in=336, looseness = 10] (2);
        \path[->,dashed,out=60,in=-60] (2) edge (3);
        \path[<-] (3) edge[out=24, in=96, looseness = 10] (3);
        \path[->] (3) edge (4);
        \path[<-] (4) edge[out=84, in=156, looseness = 10] (4);
        \path[->] (4) edge (0);

        \path[->] (0) edge[out=36, in=-36, looseness = 10] (0);
        \path[->] (0) edge (1a);
        \path[<-] (1a) edge[out=24, in=96, looseness = 10] (1a);
        \path[->] (1a) edge (2a);
        \path[<-] (2a) edge[out=84, in=156, looseness = 10] (2a);
        \path[->,dashed,out=-120,in=120] (2a) edge (4a);
        \path[<-] (4a) edge[out=204, in=276, looseness = 10] (4a);
        \path[->] (4a) edge (5a);
        \path[<-] (5a) edge[out=264, in=336, looseness = 10] (5a);
        \path[->] (5a) edge (0);
    \end{tikzpicture}
    .
\end{center}
Notice that $T$ and $\mathring{T}$ are two-edge-connected.

Now, fix $I \subseteq \set{\pm1, \ldots, \pm k}$.
Let $\rho_I$ be the partition of $V(T)$ which identifies $u_{i,1}$ with $u_{i, \ell(i)+1}$ (or $u_{i,1}'$ with $u_{i, \ell(i)+1}'$ when appropriate) for each $i \in I$, and set $T_I = T/\rho_I$.
Multiplying the diagonal matrices corresponding to the identified vertices, we get a collection $\Lambda_I$ which yields a \tdg{} $\mathring{T}_I$.

Putting this all together,
$(-1)^{\# I}\tau_{N^{\#S}}(\mathring{T}_I)$ is precisely the term in the expansion above where terms under $\Delta_{N^{\#\cS}}$ were chosen at indices corresponding to the elements of $I$.
%
Thus, we obtain
\begin{multline*}
\tr_{N^{\# \cS}}\left[
\Delta_{N^{\#\cS}}\left[(Y_k^{(N)} - \Delta_{N^{\#\cS}}[Y_k^{(N)}])^* \dots (Y_1^{(N)} - \Delta_{N^{\#\cS}}[Y_1^{(N)}])^* \right]\cdot\right.\\
\left.\Delta_{N^{\#\cS}}\left[(Y_1^{(N)} - \Delta_{N^{\#\cS}}[Y_1^{(N)}]) \dots (Y_k^{(N)} - \Delta_{N^{\#\cS}}[Y_k^{(N)}])\right] \right] \\
= \sum_{I \subseteq [k] \cup -[k]} (-1)^{|I|} \tau_{N^{\# \cS}}[\mathring{T_I}].
\end{multline*}


We can now apply Corollary \ref{cor:asymptoticexpansion2} to the test graph $\mathring{T_I}$ to obtain the following expression.

\begin{lem} \label{lem: horror}
\begin{multline*}
\mathbb{E} \norm{\Delta_{N^{\#\cS}}\left[(Y_1^{(N)} - \Delta_{N^{\#\cS}}[Y_1^{(N)}]) \dots (Y_k^{(N)} - \Delta_{N^{\#\cS}}[Y_k^{(N)}])\right]}_2^2 \\
= \sum_{I \subseteq [k] \cup -[k]} (-1)^{|I|} \sum_{\substack{\pi \in \mathcal{P}(V(T_I))^{\cS} \\ \forall s \in \cS, \pi_s \geq \rho_{I,s} \\ \forall s \in \cS, \GCC(T_I,\pi,s) \text{ is a tree}}} \lambda_N(T_I,\pi) \prod_{c \in \cC} \tau_{N^{\# \cS_c}}^\circ((T_I)_{\pi,c}) + O(N^{-1}),
\end{multline*}
where $\rho_{I,s}$ is the partition for $T_I$ defined analogously to $\rho_s$ in Notation \ref{not: rose}.
\end{lem}

We will follow a similar approach to the proof of \cite[Lemma 4.8]{ACDGM2021} and switch the order of summation above so that the terms $(-1)^{|I|}$ cancel out.
We will first express the terms on the right-hand side directly in terms of $T$ rather than in terms of $T_I$.
Since $T_I$ is a quotient of $T$, a partition $\pi_s$ of $V(T_I)$ can be pulled back to a partition $\tilde{\pi}_s$ of $V(T)$.
Moreover, a partition $\tilde{\pi}_s$ of $V(T)$ corresponds to a partition $\pi_s$ of $V(T_I)$ if and only $\pi_s \geq \rho_I$.
If $\pi \in \mathcal{P}(V(T_I))^{\cS}$ and $\tilde{\pi}$ is the corresponding element of $\mathcal{P}(V(T))^{\cS}$, then $(T_I)_{\pi,c}$ can be naturally identified with $T_{\tilde{\pi},c}$; $\GCC(T_I,\pi,s)$ is naturally identified with $\GCC(T,\tilde{\pi},s)$, and $\lambda_N(T_I,\pi) = \lambda_N(T,\tilde{\pi})$.
For verifying $\lambda_N(T_I,\pi) = \lambda_N(T,\tilde{\pi})$, the only subtlety is to understand what happens when several of the loops labeled with $\Lambda_{i,j}$'s lie over the same vertex; multiplying together the matrices labeling these self-loops is the correct way to compensate for this, which we leave as an exercise to the reader.

Thus, we can express all the terms in Lemma \ref{lem: horror} directly in terms of $T$.  Note that $\pi_s \geq \rho_{I,s}$ is equivalent to $\tilde{\pi}_s \geq \tilde{\rho}_{I,s}$, and one can check that $\tilde{\rho}_{I,s} = \rho_I \vee \rho_s$, where $\rho_s$ is the partition for $T$ given by \ref{not: rose}.  So in other words, the condition becomes $\tilde{\pi}_s \geq \rho_s$ and $\tilde{\pi}_s \geq \rho_I$.  Therefore, the expression in Lemma \ref{lem: horror} becomes
\[
\sum_{I \subseteq [k] \cup -[k]} (-1)^{|I|} \sum_{\substack{\tilde{\pi} \in \mathcal{P}(V(T))^{\cS} \\ \forall s \in \cS, \tilde{\pi}_s \geq \rho_s \vee \rho_I \\ \forall s \in \cS, \GCC(T,\tilde{\pi},s) \text{ is a tree}}} \lambda_N(T,\tilde{\pi}) \prod_{c \in \cC} \tau_{N^{\# \cS_c}}^\circ(T_{\tilde{\pi},c}) + O(1/N).
\]
Now that all everything is expressed in terms of $T$ rather than $T_I$, we can exchange the order of summation.  Here we note that $\tilde{\pi}_s \geq \rho_s \vee \rho_I$ is equivalent to $\tilde{\pi}_s \geq \rho_s$ and $\tilde{\pi}_s \geq \rho_I$.  Thus, we obtain
\[
\sum_{\substack{\tilde{\pi} \in \mathcal{P}(V(T))^{\cS} \\ \forall s \in \cS, \tilde{\pi}_s \geq \rho_s \\ \forall s \in \cS, \GCC(T,\tilde{\pi},s) \text{ is a tree}}}
\left( \sum_{\substack{I \subseteq [k] \cup -[k] \\ \rho_I \leq \bigwedge_{s \in \cS} \tilde{\pi}_s }} (-1)^{|I|} \right)
\lambda_N(T,\tilde{\pi}) \prod_{c \in \cC} \tau_{N^{\# \cS_c}}^\circ(T_{\tilde{\pi},c}) + O(1/N).
\]
Now the condition $\rho_I \leq \bigwedge_{s \in \cS} \tilde{\pi}_s$ is equivalent to $I \subseteq J_{\tilde{\pi}}$, where $J_{\tilde{\pi}}$ is the set of positive indices $i$ such that $u_{i,1}$ and $u_{i+1,1}$ are identified in $\bigwedge_{s \in \cS} \tilde{\pi}_s$ and negative indices $-i$ such that $u_{i,1}'$ and $u_{i+1,1}'$ are identified by $\bigwedge_{s \in \cS} \tilde{\pi}_s$.  Thus,
\[
\sum_{\substack{I \subseteq [k] \cup -[k] \\ \rho_I \leq \bigwedge_{s \in \cS} \tilde{\pi}_s }} (-1)^{|I|} = \sum_{I \subseteq J_{\tilde{\pi}}} (-1)^{|I|} = \begin{cases}
    1 & \text{if } J_{\tilde{\pi}} = \varnothing \\
    0 & \text{otherwise.}
\end{cases}
\]
Therefore,
\[
\sum_{\substack{\tilde{\pi} \in \mathcal{P}(V(T))^{\cS} \\ J_{\tilde{\pi}} = \varnothing \\ \forall s \in \cS, \tilde{\pi}_s \geq \rho_s \\ \forall s \in \cS, \GCC(T,\tilde{\pi},s) \text{ is a tree}}} \lambda_N(T,\tilde{\pi}) \prod_{c \in \cC} \tau_{N^{\# \cS_c}}^\circ(T_{\tilde{\pi},c}) + O(1/N).
\]
Renaming $\tilde{\pi}$ to $\pi$, we conclude the following:
\begin{lem}
Assume the setup of Theorem \ref{thm: permutation model}.  Let $T$ be the digraph in the foregoing argument (two cycles glued together).  Then
\begin{multline*}
\mathbb{E} \norm{\Delta_{N^{\#\cS}}\left[(Y_1^{(N)} - \Delta_{N^{\#\cS}}[Y_1^{(N)}]) \dots (Y_k^{(N)} - \Delta_{N^{\#\cS}}[Y_k^{(N)}])\right]}_2^2 \\
=\sum_{\substack{\pi \in \mathcal{P}(V(T))^{\cS} \\ J_\pi = \varnothing \\ \forall s \in \cS, \pi_s \geq \rho_s \\ \forall s \in \cS, \GCC(T,\pi,s) \text{ is a tree}}}
\lambda_N(T,\pi)
\prod_{c \in \cC} \tau_{N^{\# \cS_c}}^\circ(\Tpic)
+ O(1/N).
\end{multline*}
where the implicit constants in the $O(1/N)$ terms only depend on the maximum of the operator norms of the matrices $\Lambda_{i,j}$ and $X_{i,j}$ and the digraph $T$.
\end{lem}

To complete the proof of convergence in $L^2$ for Theorem \ref{thm: permutation model}, it suffices to show that there are no terms in the summation, so that only the $O(1/N)$ term remains.  To this end, we prove the following lemma, which is an analog of \cite[Lemma 4.11]{ACDGM2021} for the more complicated setting of graph products.

\begin{lem} \label{lem: inconsistency}
Let $T$ be the \tdg{} above, and adopt the notation of Theorem~\ref{thm: permutation model}.  Let $\pi \in \mathcal{P}(V(T))^{\cS}$.  Then the following conditions are inconsistent (i.e.\ cannot all be true):
\begin{enumerate}[(1)]
    \item $\pi_s \geq \rho_s$ for all $s$,
    \item $\GCC(T,\pi,s)$ is a tree for all $s$,
    \item $J_\pi = \varnothing$.
\end{enumerate}
\end{lem}

\begin{proof}[Proof of Lemma \ref{lem: inconsistency}]
  Recall that $T$ is composed of two cycles glued together at one vertex.
  In particular, $T$ is two-edge connected.
  We will in fact focus only on one of these cycles, the one with vertices $u_{1,1}$, \dots, $u_{1,\ell(1)}$ through $u_{k,1}$, \dots, $u_{k,\ell(k)}$.
  Assume that $\pi$ satisfies (1) and (2); we will show that (3) cannot hold.
  For each $i$, let $A_i \in \Comp(T_{\pi,\chi(i)})$ be the component containing $[u_{i,1}]_{\pi_{\chi(i)}}$.

  Let us consider walks induced in the graphs $\GCC(T,\pi,s)$ via Lemma~\ref{lem:inducedgccpath}.
  As the segment $u_{i,1}, \ldots, u_{i, \ell(i)}, u_{i+1,1}$ is all colored $\chi(i)$, the walk it induces in any $\GCC(T,\pi,s)$ for any $s \in \cS_{\chi(i)}$ is
  \[ [u_{i,1}]_{\pi_s}, A_i, [u_{i,2}]_{\pi_s}, A_i, \ldots, [u_{i,\ell(i)}]_{\pi_s}, A_i, [u_{i+1,1}]_{\pi_s}; \]
  in particular, $[u_{i,1}]_{\pi_s}$ and $[u_{i+1,1}]_{\pi_s}$ are both adjacent to $A_i$.
  On the other hand, if $s \notin \cS_{\chi(i)}$, then $[u_{i,1}]_{\pi_s} = [u_{i+1,1}]_{\pi_s}$.

  Suppose $s$ is fixed; then letting $\set{i_1 < \ldots < i_m} = \chi^{-1}(\cC_s)$, we have a walk in $\GCC(T,\pi,s)$ of the form
  \[ [u_{i_1,1}]_{\pi_s}, A_{i_1}, [u_{i_2,i}]_{\pi_s}, A_{i_2}, \ldots, A_{i_m}, [u_{i_1,1}]_{\pi_s}. \]
  Since $\GCC(T,\pi,s)$ is a tree, this walk must double back on itself at some point.
  We will first show that if this happens after visiting some $A_{i_j}$ (i.e., if $[u_{i_j,1}]_{\pi_s} = [u_{i_{j+1},1}]_{\pi_s}$) then $J_\pi$ is non-empty.

  Fix $j$ and suppose $u_{i_j,1} \sim_{\pi_s} u_{i_{j+1},1}$.
  We have $u_{i_j,1} \sim_{\pi_t} u_{i_j+1,1}$ for each $t \notin \cS_{\chi(i_j)}$ (since they are connected by a $\chi(i_j)$-colored path).
  Meanwhile $u_{i_j+1,1}$ is connected to $u_{i_{j+1},1}$ by a path of colors not in $\cC_s$, so we have $u_{i_j+1,1} \sim_{\pi_s} u_{i_{j+1}, 1}$ and $u_{i_j+1,1} \sim_{\pi_s} u_{i_j,1}$.
  Now $\GCC(T,\pi,s)$ is a tree, so by Lemma~\ref{lem: tree injective} we have $\hsc[\chi(i_j)]$ is injective on $A_{i_j}$; since $A_{i_j}$ contains both $[u_{i_j,1}]_{\pi_{\chi(i_j)}}$ and $[u_{i_j+1,1}]_{\pi_{\chi(i_j)}}$, it follows that they are equal.
  But if $u_{i_j,1} \sim_{\pi_{\chi(i_j)}} u_{i_j+1,1}$ then we have $u_{i_j,1} \sim_{\pi_t} u_{i_j+1,1}$ for each $t \in \cS_{\chi(i_j)}$ (since $\pi_{\chi(i_j)} \leq \pi_t$ for each $t$).
  Thus $i_j \in J_\pi$.

  So now let us suppose that the walks in $\GCC(T,\pi,s)$ above do not backtrack immediately after visiting vertices coming from colored components.
  Then each walk is either length zero, or backtracks after visiting a vertex coming from some $\Tpic$; since $A_1$ is present in at least one of these walks, they cannot all be empty and at least one backtracking must occur.
  Let $\mathcal{I} \subset [k]\times[k]$ be the set of all tuples $(x, y)$ for which $x < y$ and $A_x = A_y$.
  This set is non-empty as just noted, so we may take $(x, y) \in \mathcal{I}$ with $y-x$ minimal.
  Since $\chi(1)\cdots\chi(k)$ is $G$-reduced, there is some $x < z < y$ so that $\chi(x)$ and $\chi(z)$ share a common string $s \in \cS$.
  Recalling the notation of the previous paragraph, let $\set{i_1 < \ldots < i_m} = \chi^{-1}(\cC_s)$.

  Some sub-path of
  \[ [u_{i_1,1}]_{\pi_s}, A_{i_1}, [u_{i_2,i}]_{\pi_s}, A_{i_2}, \ldots, A_{i_m}, [u_{i_1,1}]_{\pi_s} \]
  visits $A_x$, then $A_z$, then $A_y = A_x$ again.
  This sub-path cannot backtrack immediately after visiting $A_z$ for the first time.
  However it does return to $A_y$, and so must pass through $A_z$ again (for $\GCC(T,\pi,s)$ is a tree); then there is some $z < z' < y$ so that $A_z = A_{z'}$.
  But then $(z, z') \in \mathcal{I}$ and $z'-z < y-x$, contradicting the choice of $(x,y)$.
\end{proof}

\subsection{Almost sure convergence} \label{subsec: almost sure}

In this section, we upgrade from convergence in probability to almost sure convergence and complete the proof of Theorem~\ref{thm: permutation model}.  We use the following concentration result due to Maurey \cite{MaureyPerm} and is given in \cite[\S 7.5-7.6]{MS1986}.  See also \cite[Lemma 4.6]{DKP2014}.

\begin{lem} \label{lem: concentration}
Consider the symmetric group $S_N$ with the metric
\[
d(\sigma,\pi) = \frac{1}{N} \#\{i: \sigma(i) \neq \pi(i) \}.
\]
Let $\Sigma$ be a uniformly random element of $S_N$.  Let $f: S_N \to \bR$ be $L$-Lipschitz.  Then
\[
P(|f(\Sigma) - \mathbb{E} f(\Sigma)| \geq \epsilon) \leq 2e^{-N \epsilon^2/ 64 L^2}.
\]
\end{lem}



With this in hand, we are ready to complete the proof of Theorem~\ref{thm: permutation model}.

\begin{proof}[Proof of Theorem~\ref{thm: permutation model}]
Fix $N$.  Let $f_N: \prod_{c \in \cC} S_{N^{\# \cS_c}} \to \bR$ be the function mapping a tuple of permutations $(\Sigma_c^{(N)})_{c \in \cC}$ to
\[
\norm{\Delta_{N^{\#\cS}}[(Y_1^{(N)} - \Delta_{N^{\#\cS}}[Y_1^{(N)}]) \dots (Y_k^{(N)} - \Delta_{N^{\#\cS}}[Y_k^{(N)}])]}_2^2.
\]
Here the dependence on $\Sigma_c^{(N)}$ occurs through the terms $\uu{X}_{i,j}^{(N)}$. As before, let $\sigma_c\in S_N$ denote the permutation corresponding to $\Sigma_c^{(N)}$. We claim that $f_N$ is Lipschitz in each argument $\Sigma_c^{(N)}$ with Lipschitz constant independent of $N$.  To see this, first note that
    \[
        \|\Sigma_c^{(N)} - \Sigma_{c'}^{(N)}\|_1 \leq \| \Sigma_c^{(N)} - \Sigma_{c'}^{(N)}\|_2 = \sqrt{2} d(\sigma_c, \sigma_{c'})^{\frac12}.
    \]
Next, because the $\Lambda_{i,j}^{(N)}$'s and the $X_{i,j}^{(N)}$'s are assumed to be uniformly bounded in operator norm, the mapping from $\Sigma_c^{(N)}$ to $Y_i^{(N)}$ is Lipschitz with respect to the $d$ metric in the domain and the $L^1$ norm in the target space. 
For all $A,C\in M_{N^{\# S}}(\bC)$ we have:
\begin{equation}\label{eqn: duality 1}
\sup_{B\in M_{N^{\# \cS}}(\bC):\|B\|\leq 1}|\tr(AB)|=\|A\|_{1}, \textnormal{ and } \tr(\Delta_{N^{\#\cS}}(A)C)=\tr(A\Delta_{N^{\#\cS}}(C)).
\end{equation}
Using (\ref{eqn: duality 1}) and $\|\Delta_{N^{\#S}}(C)\|\leq \|C\|$ for all $C\in M_{N^{\#S}}(\bC)$, we obtain that $\Delta_{N^{\# \cS}}$ is contractive in $L^{1}$.
Thus, we can also say that $Y_i^{(N)} - \Delta_{N^{\#\cS}}(Y_i^{(N)})$ depends on $\Sigma_c^{(N)}$ in a Lipschitz manner.  
Finally, again using the uniform boundedness in operator norm, we get Lipschitzness of $\Delta_{N^{\#\cS}}[(Y_1^{(N)} - \Delta_{N^{\#\cS}}[Y_1^{(N)}]) \dots (Y_k^{(N)} - \Delta_{N^{\#\cS}}[Y_k^{(N)}])]$.  Similarly,
\begin{multline*}
\tr_{N^{\#\cS}} \left[ \Delta_{N^{\#\cS}}[(Y_1^{(N)} - \Delta_{N^{\#\cS}}[Y_1^{(N)}]) \dots (Y_k^{(N)} - \Delta_{N^{\#\cS}}[Y_k^{(N)}])]^*\cdot\right.\\
\left. \Delta_{N^{\#\cS}}[(Y_1^{(N)} - \Delta_{N^{\#\cS}}[Y_1^{(N)}]) \dots (Y_k^{(N)} - \Delta_{N^{\#\cS}}[Y_k^{(N)}])] \right]
\end{multline*}
is Lipschitz in $\Sigma_c^{(N)}$ with Lipschitz constant $L_c$ independent of $N$.

Now, let $c_1$, \dots, $c_m$ be the colors in $\cC$.  Let $\mathbb{E}_{c_j}$ be the conditional expectation obtained by integrating out the $\Sigma_{c_j}$.  Since the $\Sigma_{c_j}$'s are independent, by Lemma \ref{lem: concentration} we have
\[
P(|f(\Sigma_{c_1}^{(N)},\dots,\Sigma_{c_m}^{(N)}) - \mathbb{E}_{c_1}[f(\Sigma_{c_1}^{(N)},\dots,\Sigma_{c_m}^{(N)})]| \geq \epsilon/m) \leq 2e^{-N^{\# \cS_{c_1}}\epsilon^2/64m^2L_{c_1}^2}.
\]
Independence also implies that $\mathbb{E}_{c_1}[f(\Sigma_{c_1}^{(N)},\dots,\Sigma_{c_m}^{(N)})]$ is $L_{c_2}$-Lipschitz with respect to $\Sigma_{c_2}^{(N)}$.  Hence, we can apply the concentration estimate for $\Sigma_{c_2}^{(N)}$ to this function.  Continuing this procedure iteratively, we obtain
\[
P(|f_N(\Sigma_{c_1}^{(N)},\dots,\Sigma_{c_m}^{(N)}) - \mathbb{E}[f_N(\Sigma_{c_1}^{(N)},\dots,\Sigma_{c_m}^{(N)})]| \geq \epsilon) \leq 2 \sum_{j=1}^m e^{-N^{\# \cS_{c_j}}\epsilon^2/64m^2L_{c_j}^2}.
\]
The right-hand side is summable in $N$, and so the Borel-Cantelli lemma implies that
\[
\lim_{N \to \infty} |f_N(\Sigma_{c_1}^{(N)},\dots,\Sigma_{c_m}^{(N)}) - \mathbb{E}[f_N(\Sigma_{c_1}^{(N)},\dots,\Sigma_{c_m}^{(N)})]| = 0
\]
almost surely.
Now in \S\ref{subsec: asym graph ind over diag} we showed that
\[
\lim_{N \to \infty} \mathbb{E}[f_N(\Sigma_{c_1}^{(N)},\dots,\Sigma_{c_m}^{(N)})] = 0.
\]
Hence, $f_N(\Sigma_{c_1}^{(N)},\dots,\Sigma_{c_m}^{(N)}) \to 0$ almost surely as desired.
\end{proof}

\section{Application to sofic groups} \label{subsec: sofic}

First, we recall the original definition of soficity.  For permutations $\sigma$ and $\tau$ in $S_N$, the \emph{Hamming distance} is given by
\[
d_{\operatorname{Hamm}}(\sigma,\tau) = \frac{1}{N} \# \{j: \tau(j) \neq \sigma(j)\}.
\]
If $\Sigma$ and $T$ are the permutation matrices representing $\sigma$ and $\tau$, then one can easily check that
\[
d_{\operatorname{Hamm}}(\sigma,\tau) = \tr_N(I - \Sigma^* T).
\]

\begin{defn}
A countable group $\Gamma$ is said to be \emph{sofic} if there exists a sequence of natural numbers $N_k \to \infty$ and a sequence of maps $\sigma_k: \Gamma \to S_{N_k}$ such that
\begin{equation} \label{eq: sofic assumption 1}
\lim_{k \to \infty} d_{\operatorname{Hamm}}(\sigma_k(g) \sigma_k(h), \sigma_k(gh)) = 0 \text{ for } g, h \in \Gamma
\end{equation}
and
\begin{equation} \label{eq: sofic assumption 2}
\lim_{k \to \infty} d_{\operatorname{Hamm}}(\sigma_k(g),\id) = 1 - \delta_{g = e} \text{ for } g \in \Gamma.
\end{equation}
We call such a sequence $(\sigma_k)_{k \in \bN}$ an \emph{asymptotic homomorphism into permutation groups}.
\end{defn}

This definition can be restated in terms of matrices as follows.

\begin{prop}[Folklore] \label{prop: sofic}
Let $\Gamma$ be a countable group generated by $(g_j)_{j \in \bN}$; for notational convenience let $g_{-j} = g_j^{-1}$ for $j \in \bN$.  Then $\Gamma$ is sofic if and only if there exist permutation matrices $(T_j^{(N)})_{j,N \in \bN}$ such that for all $m \in \bN$ and $j(1)$, \dots, $j(m) \in \bZ \setminus \{0\}$, we have
\begin{equation} \label{eq: sofic matrix version}
\lim_{N \to \infty} \tr_N(T_{j(1)}^{(N)} \dots T_{j(m)}^{(N)}) = \delta_{g_{j(1)} \dots g_{j(m)}=e},
\end{equation}
where we also write $T_{-j}^{(N)} = (T_j^{(N)})^{-1}$.  Furthermore, it suffices for there to exist some matrices $T_j^{(N)}$ for a subsequence of natural numbers $N(k) \to \infty$.
\end{prop}

\begin{proof}
First, assume that $\Gamma$ is sofic, and let $\sigma_k: \Gamma \to S_{n(k)}$ be an asymptotic homomorphism.  First, to translate the conditions \eqref{eq: sofic assumption 1} and \eqref{eq: sofic assumption 2} into statements about words in the generators, we will show that for any $j(1)$, \dots, $j(m) \in \bZ \setminus \{0\}$, we have
\[
\lim_{k \to \infty} d_{\operatorname{Hamm}}(\sigma_k(g_{j(1)}) \dots \sigma_k(g_{j(m)}),\id) = 1 - \delta_{g_{j(1)} \dots g_{j(m)}=e}.
\]
To see this, note that asymptotic multiplicativity implies that 
\[\lim_{k\to\infty}d_{\operatorname{Hamm}}(\sigma_k(g_{j(1)}) \dots \sigma_k(g_{j(m)}),\sigma_{k}(g_{j(1)}\cdots g_{j(m)})) =0\]
for all $j(1),\cdots,j(m)\in \bZ\setminus\{0\}$ (by inducting on $m$ and applying the triangle inequality and invariance of the Hamming distance). Thus our claim is equivalent to the statement that
\[
\lim_{k \to \infty} d_{\operatorname{Hamm}}(\sigma_k(g_{j(1)}g_{j(2)}\cdots g_{j(m)}),\id) = 1 - \delta_{g_{j(1)} \dots g_{j(m)}=e},
\]
which is true by \eqref{eq: sofic assumption 2}.

Second, we want to arrange that the permutation models have size exactly $N$ rather than $n_k$.  Fix a sequence $k(N)$ such that $\lim_{N \to \infty} k(N) = \infty$ and $\lim_{k \to \infty} n(k(N)) / N = 0$.  Write $N = q(N) n(k(N)) + r(N)$ for integers $q(N)$ and $r(N) \in \{0,\dots, n(k(N)) - 1\}$.  Let $\Sigma_k(g)$ be the matrix associated to the permutation $\sigma_k(g)$, and let $T_j^{(N)}$ be the block diagonal permutation matrix
\[
T_j^{(N)} = \Sigma_k(g_j)^{\oplus q(N)} \oplus I_{r(N)} = \underbrace{\Sigma_k(g_j) \oplus \dots \oplus \Sigma_k(g_j)}_{q(N)} \oplus I_{r(N)}.
\]
Let $j(1)$, \dots, $j(m)$ be given and note
\begin{align*}
\Tr_N(T_{j(1)}^{(N)} \dots T_{j(m)}^{(N)}) &= \Tr_N(\Sigma_{k(N)}(g_{j(1)})^{\oplus q(N)} \dots \Sigma_{k(N)}(g_{j(m)}) \oplus I_{r(N)}) \\
&= q(N) \Tr_{n(k(N))}(\Sigma_{k(N)}(g_{j(1)})^{\oplus q(N)} \dots \Sigma_{k(N)}(g_{j(m)})) + r(N) \\
&= q(N) n(k(N)) (1 - d_{\operatorname{Hamm}}(\sigma_{k(N)}(g_{j(1)})) \dots \sigma_{k(N)}(g_{j(m)}),\id) + r(N),
\end{align*}
and hence
\begin{align*}
\tr_N(T_{j(1)}^{(N)} \dots T_{j(m)}^{(N)}) &= \frac{q(N) n(k(N))}{N} (1 - d_{\operatorname{Hamm}}(\sigma_{k(N)}(g_{j(1)}) \dots \sigma_{k(N)}(g_{j(m)}),\id)) + \frac{r(N)}{N} \\
&= 1 - d_{\operatorname{Hamm}}(\sigma_{k(N)}(g_{j(1)}) \dots \sigma_{k(N)}(g_{j(m)}),\id) \\
& \quad + \frac{r(N)}{N} d_{\operatorname{Hamm}}(\sigma_{k(N)}(g_{j(1)}) \dots \sigma_{k(N)}(g_{j(m)}),\id).
\end{align*}
Since $r(N) < n(N)$, we have $r(N) / N \to 0$, and therefore,
\[
\lim_{N \to \infty} \tr_N(T_{j(1)}^{(N)} \dots T_{j(m)}^{(N)}) = \delta_{g_{j(1)} \dots g_{j(m)} = e}.
\]

Conversely, assume we are given $T_j^{(N)}$ satisfying \eqref{eq: sofic matrix version}.  Let $\sigma_{N,j} \in S_N$ be the permutation corresponding to $T_j^{(N)}$.  First, note that for two words $g_{j(1)} \dots g_{j(m)}$ and $g_{j'(1)} \dots g_{j'(m')}$, we have
\begin{align*}
d_{\operatorname{Hamm}}(\sigma_{N,j(1)} \dots \sigma_{N,j(m)}, \sigma_{N,j'(1)} \dots \sigma_{N,j'(m')}) &= \tr_N(I - T_{-j(m)}^{(N)} \dots T_{-j(1)}^{(N)} T_{j'(1)}^{(N)} \dots T_{j'(m')}^{(N)}) \\
&= 1 - \delta_{g_{-j(m)} \dots g_{-j(1)} g_{j'(1)} \dots g_{j'(m')}=e} \\
&= 1 - \delta_{g_{j(1)} \dots g_{j(m)}=g_{j'(1)} \dots g_{j'(m')}}
\end{align*}
For each group element $g$, choose an expression of $g$ as a product of generators $g = g_{j(1)} \dots g_{j(m)}$, and define $\sigma_N(g)$ as the permutation given by the matrix $T_{j(1)}^{(N)} \dots T_{j(m)}^{(N)}$.  Applying the previous equation with $g_{j'(1)} \dots g_{j'(m')}$ equal to the empty word shows that \eqref{eq: sofic assumption 2} holds.  To show \eqref{eq: sofic assumption 1}, note that multiplying our chosen expressions for $g$ and $h$ as products of the generators gives one possible expression for $gh$ as a product of the generators, which does not necessarily agree with the expression for $gh$ as a product of the generators that we chose at the beginning.  But the previous equation implies that the Hamming distance between two expressions for $gh$ as a product of the generators converges to zero, and so \eqref{eq: sofic assumption 1} holds.  For the last claim, note that the proof of the converse direction works equally well with a subsequence $N(k)$ rather than all natural numbers.
\end{proof}

Now we explain how Theorem \ref{thm: permutation model} implies that soficity is preserved by graph products of group.

\begin{proof}[Proof of Proposition \ref{prop: sofic graph product}]
Let $\mathcal{G}$ be given and let $(\Gamma_c)_{c \in \mathcal{C}}$ be countable sofic groups.  For each $c \in \cC$, fix a generating set $(g_{c,j})_{j \in \cN}$ for the group $\Gamma_c$, and fix a sequence of permutation models $(T_{c,j}^{(N)})_{j,N \in \bN}$ as in Proposition \ref{prop: sofic}.  Fix an assignment of strings $\cS$ and uniformly random permutations $\Sigma_c^{(N)}$ as in Theorem \ref{thm: permutation model}.  For each $c \in \cC$, fix an identification $M_{N^{\#\cS_c}}(\bC) \cong \bigotimes_{s \in \cS_c} M_N(\bC)$ in the usual way using decompositions into block matrices.  Let
\[
Z_{c,j}^{(N)} = (\Sigma_c^{(N)})^* T_{c,j}^{(N^{\# \cS_c})} \Sigma_c^{(N)} \otimes I_{N^{\# \S \setminus \cS_c}} \in \bigotimes_{s \in \cS_c} M_N(\bC) \otimes I_{N^{\# \S \setminus \cS_c}} \subseteq \bigotimes_{s \in \cS} M_N(\bC).
\]
As usual, write $Z_{c,-j}^{(N)} = (Z_{c,j}^{(N)})^*$.  We claim that almost surely $(Z_{c,j}^{(N)})_{c \in \cC, j \in \bN}$ is a sequence of $N^{\# \cS} \times N^{\# \cS}$ permutation approximations for the generating set $(g_{c,j})_{c \in \cC, j \in \bN}$ of the graph product group $\Gamma = \gp_{c \in \mathcal{G}} \Gamma_c$.  Thus, we must show that almost surely
\begin{equation} \label{eq: soficity goal}
\lim_{N \to \infty} \tr_{N^{\# S}}[Z_{\chi(1),j(1)}^{(N)} \dots Z_{\chi(m),j(m)}^{(N)}] = \delta_{g_{\chi(1),j(1)} \dots g_{\chi(m),j(m)} = e}.
\end{equation}
We proceed by induction on $m$.  First, we group together adjacent terms where the color $\chi(i)$ is the same.  Moreover, since $Z_{c,j}^{(N)}$ commutes with $Z_{c',j'}^{(N)}$ when $c \sim c'$ in $\mathcal{G}$, we may apply commutation operations and group terms together so that any two separated groups of terms from some color $c$ must have some term from $c'$ between them where $c'$ does not commute with $c$.  Thus, after suitable changes notation, the original product can be rewritten as
\[
(Z_{\chi(1),j(1,1)}^{(N)} \dots Z_{\chi(1),j(1,\ell(1))}^{(N)}) \dots (Z_{\chi(k),j(k,1)}^{(N)} \dots Z_{\chi(1),j(k,\ell(k))}^{(N)})
\]
where $m = \ell(1) + \dots + \ell(k)$ and where $\chi(1)$, \dots, $\chi(k)$ forms a $\mathcal{G}$-reduced word.

Now we consider two cases.  First, suppose there is some $i$ such that $g_{\chi(i),j(i,1)} \dots g_{\chi(i),j(i,\ell(i))} = e$ in $G_{\chi(i)}$.  Then by our assumption,
\[
\lim_{N \to \infty} \tr_N[T_{\chi(i),j(i,1)}^{(N)} \dots T_{\chi(i),j(i,\ell(i))}^{(N)}] = 1.
\]
Since $T_{\chi(i),j(i,1)}^{(N)} \dots T_{\chi(i),j(i,\ell(i))}^{(N)}$ is a permutation matrix, hence also a unitary matrix with real trace, we have
\[
\norm{I_N - T_{\chi(i),j(i,1)}^{(N)} \dots T_{\chi(i),j(i,\ell(i))}^{(N)}}_2^2 = 2 - 2 \tr_N[T_{\chi(i),j(i,1)}^{(N)} \dots T_{\chi(i),j(i,\ell(i))}^{(N)}] \to 0.
\]
Now $Z_{\chi(i),j(i,i')}$ is obtained by conjugation with $\Sigma_{\chi(i)}^{(N)}$ and tensoring with the identity matrix, and so also
\[
\norm{I_{N^{\# \cS}} - Z_{\chi(i),j(i,1)}^{(N)} \dots Z_{\chi(i),j(i,\ell(i))}^{(N)}}_2^2 = \norm{I_N - T_{\chi(i),j(i,1)}^{(N)} \dots T_{\chi(i),j(i,\ell(i))}^{(N)}}_2^2 \to 0.
\]
In the larger product, $(Z_{\chi(1),j(1,1)}^{(N)} \dots Z_{\chi(1),j(1,\ell(1))}^{(N)}) \dots (Z_{\chi(k),j(k,1)}^{(N)} \dots Z_{\chi(1),j(k,\ell(k))}^{(N)})$, all the terms are unitary, hence bounded by $1$ in operator norm.  Thus, if we swap out $Z_{\chi(i),j(i,1)}^{(N)} \dots Z_{\chi(i),j(i,\ell(i))}^{(N)}$ for $I_{N^{\# \cS}}$, the error is bounded in $2$-norm by $\norm{I_{N^{\# \cS}} - Z_{\chi(i),j(i,1)}^{(N)} \dots Z_{\chi(i),j(i,\ell(i))}^{(N)}}_2$, and hence the difference in the trace also approaches zero.  Therefore, it suffices to prove \eqref{eq: soficity goal} for the product of terms $(Z_{\chi(1),j(1,1)}^{(N)} \dots Z_{\chi(1),j(1,\ell(1))}^{(N)}) \dots (Z_{\chi(k),j(k,1)}^{(N)} \dots Z_{\chi(1),j(k,\ell(k))}^{(N)})$ with $Z_{\chi(i),j(i,1)}^{(N)} \dots Z_{\chi(i),j(i,\ell(i))}^{(N)}$ removed.  This is a product with a strictly smaller total number of terms (the original number of terms being $m = \ell(1) + \dots \ell(k)$).  Therefore, we can apply the induction hypothesis.

Hence, it remains to handle the case when $g_{\chi(i),j(i,1)} \dots g_{\chi(i),j(i,\ell(i))} \neq e$ for every $i$.  Here we apply Theorem \ref{thm: permutation model} with
\begin{align*}
\uu{X}_{i,i'}^{(N)} &= Z_{\chi(i),j(i,i')} \\
\Lambda_{i,i'}^{(N)} &= I_{N^{\# \cS}} \\
Y_i^{(N)} &= Z_{\chi(i),j(i,1)}^{(N)} \dots Z_{\chi(i),j(i,\ell(i))}^{(N)},
\end{align*}
and thus obtain
\begin{equation} \label{eq: product limit equation}
  \lim_{N \to \infty} \norm{\Delta_{N^{\#\cS}}[(Y_1^{(N)} - \Delta_{N^{\#\cS}}[Y_1^{(N)}]) \dots (Y_k^{(N)} - \Delta_{N^{\#\cS}}[Y_k^{(N)}])]}_2 = 0 \text{ almost surely.}
\end{equation}
Now recall $Y_i^{(N)}$ is a permutation matrix and $\lim_{N \to \infty} \tr_{N^{\# \cS}}(Y_i^{(N)}) = 0$ because $Y_i^{(N)}$ corresponds to a non-identity word in $G_{\chi(i)}$.  Note $\Delta_{N^{\# \cS}}[Y_i^{(N)}]$ is a diagonal matrix of zeros and ones, and hence
\[
\norm{\Delta_{N^{\# \cS}}[Y_i^{(N)}]}_2 = \tr_{N^{\# \cS}}[\Delta_{N^{\# \cS}}[Y_i^{(N)}]^2]^{1/2} = \tr_{N^{\# \cS}}[\Delta_{N^{\# \cS}}[Y_i^{(N)}]]^{1/2} \to 0.
\]
Note each of the terms in \eqref{eq: product limit equation} is bounded in operator norm, and hence the errors from replacing $\Delta_{N^{\# \cS}}[Y_i^{(N)}]$ by $0$ will be bounded by a constant times $\norm{\Delta_{N^{\# \cS}}[Y_i^{(N)}]}_2$.  Therefore, we get
\[
  \lim_{N \to \infty} \norm{\Delta_{N^{\#\cS}}[Y_1^{(N)} \dots Y_k^{(N)} ]}_2 = 0 \text{ almost surely.}
\]
In particular,
\[
|\tr_{N^{\# \cS}}[Y_1^{(N)} \dots Y_k^{(N)}]| = |\tr_{N^{\#\cS}}[\Delta_{N^{\#\cS}} [Y_1^{(N)} \dots Y_k^{(N)}]]| \leq \norm{\Delta_{N^{\#\cS}}[Y_1^{(N)} \dots Y_k^{(N)} ]}_2 \to 0 \text{ almost surely.}
\]
Thus,
\[
\tr_{N^{\# \cS}}[(Z_{\chi(1),j(1,1)}^{(N)} \dots Z_{\chi(1),j(1,\ell(1))}^{(N)}) \dots (Z_{\chi(k),j(k,1)}^{(N)} \dots Z_{\chi(1),j(k,\ell(k))}^{(N)})] \to 0 \text{ almost surely.}
\]
This establishes \eqref{eq: soficity goal} in this case, and thus completes the induction proof.

Therefore, the matrices $Z_{c,j}^{(N)}$ are almost surely a permutation approximation for $\gp_{c \in \cG} \Gamma_c$ in the sense of Proposition \ref{prop: sofic}, which shows that $\gp_{c \in \cG} \Gamma_c$ is sofic.
\end{proof}

\appendix

\section{Notes}
\label{apndx: notation}

\subsection{An example with pictures}
\label{asubsec:tgraph ex}

Here we work through an extensive example to demonstrate the various graphs used in Subsections~\ref{subsec: combinatorial graph moments} and \ref{subsec: growth graphs}.

Let $\cC = \set{\cb, \cg, \crr}$ and $\cS = \set{1,2,3}$, and define $\sstrc$ by $1 \sstrc \cb$, $2 \sstrc \cb, \cg$, and $3 \sstrc \cg, \crr$.
Then
\begin{align*}
  \cC_1 &= \set{\cb}
  & \cC_2 &= \set{\cb, \cg}
  & \cC_3 &= \set{\cg, \crr} \\
  \cS_\cb &= \set{1,2}
  & \cS_\cg &= \set{2, 3}
  & \cS_\crr &= \set{3}
\end{align*}
In the visualization of \cite[\S3.2]{CC2021} this corresponds to operators living on strings in the following manner:
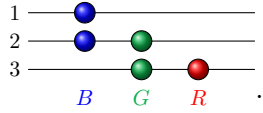
\[\begin{tikzpicture}[baseline, scale=.75, every node/.style={scale=.75}]

		\draw (0,3/2) node [left] {$1$} -- (4,3/2);
		\draw (0,2/2) node [left] {$2$} -- (4,2/2);
		\draw (0,1/2) node [left] {$3$} -- (4,1/2);

    \def\shades{{0, "blue", "green", "red"}}
    \def\chords{{"{0}", "{1, 2}", "{2,3}", "{3}"}}
    \def\labels{{"0", "B", "G", "R"}}
		\def\sequence{{0,1,2,3}}

		\foreach \x in {1, ..., 3} {
			\pgfmathparse{\sequence[\x]}
			\edef\ind{\pgfmathresult}
			\pgfmathparse{\chords[\ind]}
			\edef\chord{\pgfmathresult}
			\pgfmathparse{\shades[\ind]}
			\edef\myshade{\pgfmathresult}
			\foreach \y in \chord {
				\node [draw, shade, circle, ball color=\myshade] at (\x, 2-0.5*\y) {};
			}
      \pgfmathparse{\labels[\ind]}
      \edef\lbl{\pgfmathresult}
      \node [color=\myshade] (a) at (\x, 0) {$\lbl$};
		}
\end{tikzpicture}.\]

Define $T$ to be the following \tdg:
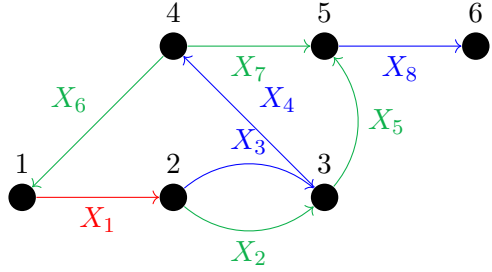
\[\begin{tikzpicture}
    \node[circle,fill,label=1] (1) at (0,0) {};
    \node[circle,fill,label=2] (2) at (2,0) {};
    \node[circle,fill,label=3] (3) at (4,0) {};
    \node[circle,fill,label=4] (4) at (2,2) {};
    \node[circle,fill,label=5] (5) at (4,2) {};
    \node[circle,fill,label=6] (6) at (6,2) {};

    \path[->,color=red] (1) edge node[auto,swap] {$X_1$} (2);
    \path[->,color=green,in=220,out=320] (2) edge node[auto,swap] {$X_2$} (3);
    \path[->,color=blue,in=140,out=40] (2) edge node[auto,swap,above] {$X_3$} (3);
    \path[->,color=blue] (3) edge node[auto,swap] {$X_4$} (4);
    \path[->,color=green,out=50,in=310] (3) edge node[auto,swap] {$X_5$} (5);
    \path[->,color=green] (4) edge node[auto,swap] {$X_6$} (1);
    \path[->,color=green] (4) edge node[auto,swap] {$X_7$} (5);
    \path[->,color=blue] (5) edge node[auto,swap] {$X_8$} (6);
\end{tikzpicture}\]
(Explicitly: we are taking $\chi(1) = {\color{red}R}$, $\chi(2) = \chi(5) = \chi(6) = \chi(7) = {\color{green}G}$, and $\chi(3) = \chi(4) = \chi(8) = {\color{blue}B}$.)

We then have the following, from Notation~\ref{not: rose}:
\[
  \rho_1 = \set{ \{1,2,3,4,5\}, \{6\} }, \qquad
  \rho_2 = \set{ \{1,2\}, \{3\}, \{4\}, \{5\}, \{6\} }, \qquad
  \rho_3 = \set{ \{1\}, \{2,3,4\}, \{5,6\} }. \qquad
\]

To keep things slightly simpler, we will take $\pi = \rho$.
In the language of Notation~\ref{not:omegapic}, we have
\[
  \pi_{\cb} = \set{ \{1,2\}, \{3\}, \{4\}, \{5\}, \{6\} }, \qquad
  \pi_{\cg} = \set{ \{1\}, \{2\}, \{3\}, \{4\}, \{5\}, \{6\} }, \qquad
  \pi_{\crr} = \set{ \{1\}, \{2,3,4\}, \{5,6\} }.
\]

This leads to the following three digraphs:
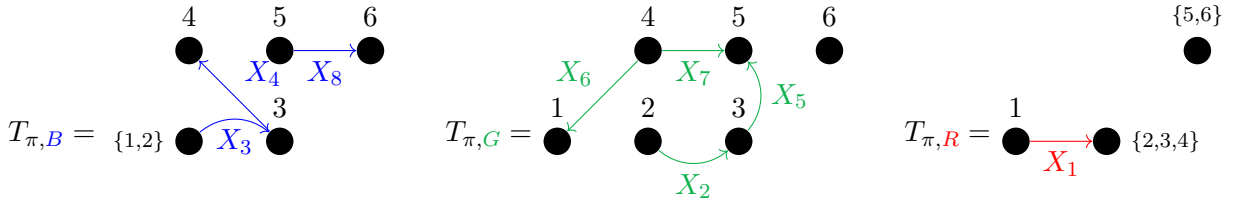
\[
  T_{\pi,\cb} =
  \begin{tikzpicture}[baseline,scale=.6]
    \node[circle,fill,label=left:{\scriptsize{\{1,2\}}}] (1) at (2,0) {};
    \node[circle,fill,label=3] (3) at (4,0) {};
    \node[circle,fill,label=4] (4) at (2,2) {};
    \node[circle,fill,label=5] (5) at (4,2) {};
    \node[circle,fill,label=6] (6) at (6,2) {};

    \path[->,color=blue,in=140,out=40] (1) edge node[auto,swap] {$X_3$} (3);
    \path[->,color=blue] (3) edge node[auto,swap] {$X_4$} (4);
    \path[->,color=blue] (5) edge node[auto,swap] {$X_8$} (6);
  \end{tikzpicture}
\qquad
  T_{\pi,\cg} =
  \begin{tikzpicture}[baseline,scale=.6]
    \node[circle,fill,label=1] (1) at (0,0) {};
    \node[circle,fill,label=2] (2) at (2,0) {};
    \node[circle,fill,label=3] (3) at (4,0) {};
    \node[circle,fill,label=4] (4) at (2,2) {};
    \node[circle,fill,label=5] (5) at (4,2) {};
    \node[circle,fill,label=6] (6) at (6,2) {};

    \path[->,color=green,in=220,out=320] (2) edge node[auto,swap] {$X_2$} (3);
    \path[->,color=green,out=50,in=310] (3) edge node[auto,swap] {$X_5$} (5);
    \path[->,color=green] (4) edge node[auto,swap] {$X_6$} (1);
    \path[->,color=green] (4) edge node[auto,swap] {$X_7$} (5);
\end{tikzpicture}
\qquad
  T_{\pi,\crr} =
  \begin{tikzpicture}[baseline,scale=.6]
    \node[circle,fill,label=1] (1) at (0,0) {};
    \node[circle,fill,label=right:{\scriptsize{\{2,3,4\}}}] (2) at (2,0) {};
      \node[circle,fill,label={\scriptsize{\{5,6\}}}] (5) at (4,2) {};

    \path[->,color=red] (1) edge node[auto,swap] {$X_1$} (2);
\end{tikzpicture}
\]
Next, the graphs introduced in Notation~\ref{not:gpis} are as follows:
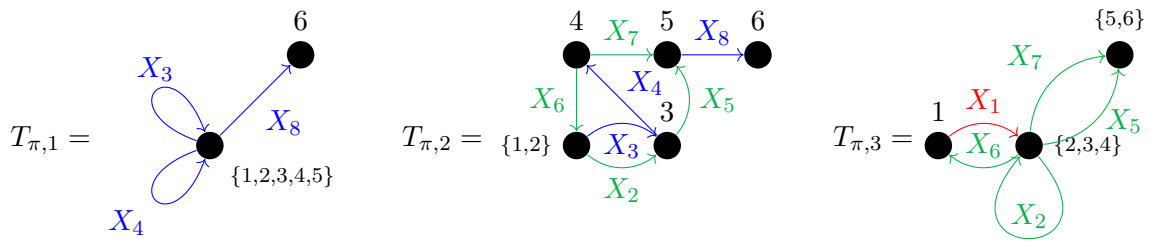
\[
  T_{\pi,1} =
  \begin{tikzpicture}[baseline,scale=.6]
    \node[circle,fill,label=below right: {\scriptsize{\{1,2,3,4,5\}}}] (1) at (0,0) {};
    \node[circle,fill,label=6] (6) at (2,2) {};
    \path[->,color=blue,in=110,out=160,looseness=20] (1) edge node[auto,swap,above] {$X_3$} (1);
    \path[->,color=blue,in=250,out=200,looseness=20] (1) edge node[auto,swap] {$X_4$} (1);
    \path[->,color=blue] (1) edge node[auto,swap] {$X_8$} (6);
  \end{tikzpicture}
  \qquad
  T_{\pi,2} =
  \begin{tikzpicture}[baseline,scale=.6]
    \node[circle,fill,label=left:{\scriptsize{\{1,2\}}}] (1) at (2,0) {};
    \node[circle,fill,label=3] (3) at (4,0) {};
    \node[circle,fill,label=4] (4) at (2,2) {};
    \node[circle,fill,label=5] (5) at (4,2) {};
    \node[circle,fill,label=6] (6) at (6,2) {};

    \path[->,color=green,in=220,out=320] (1) edge node[auto,swap] {$X_2$} (3);
    \path[->,color=blue,in=140,out=40] (1) edge node[auto,swap] {$X_3$} (3);
    \path[->,color=blue] (3) edge node[auto,swap,above right=-.1cm] {$X_4$} (4);
    \path[->,color=green,out=50,in=310] (3) edge node[auto,swap] {$X_5$} (5);
    \path[->,color=green] (4) edge node[auto,swap] {$X_6$} (1);
    \path[->,color=green] (4) edge node[auto,swap,above] {$X_7$} (5);
    \path[->,color=blue] (5) edge node[auto,swap,above] {$X_8$} (6);
  \end{tikzpicture}
  \qquad
  T_{\pi,3} =
  \begin{tikzpicture}[baseline,scale=.6]
    \node[circle,fill,label=1] (1) at (0,0) {};
    \node[circle,fill,label=right:{\scriptsize{\{2,3,4\}}}] (2) at (2,0) {};
    \node[circle,fill,label={\scriptsize{\{5,6\}}}] (5) at (4,2) {};

    \path[->,color=red,in=140,out=40] (1) edge node[auto,swap,above] {$X_1$} (2);
    \path[->,color=green,in=230,out=310,looseness=20] (2) edge node[auto,swap] {$X_2$} (2);
    \path[->,color=green,out=5,in=265] (2) edge node[auto,swap, right] {$X_5$} (5);
    \path[->,color=green,in=320,out=220] (2) edge node[auto,swap] {$X_6$} (1);
    \path[->,color=green,in=185,out=85] (2) edge node[auto,swap,above left] {$X_7$} (5);
  \end{tikzpicture}
\]

Finally, we can consider the graphs of colored components from Definition~\ref{def:gcc}:
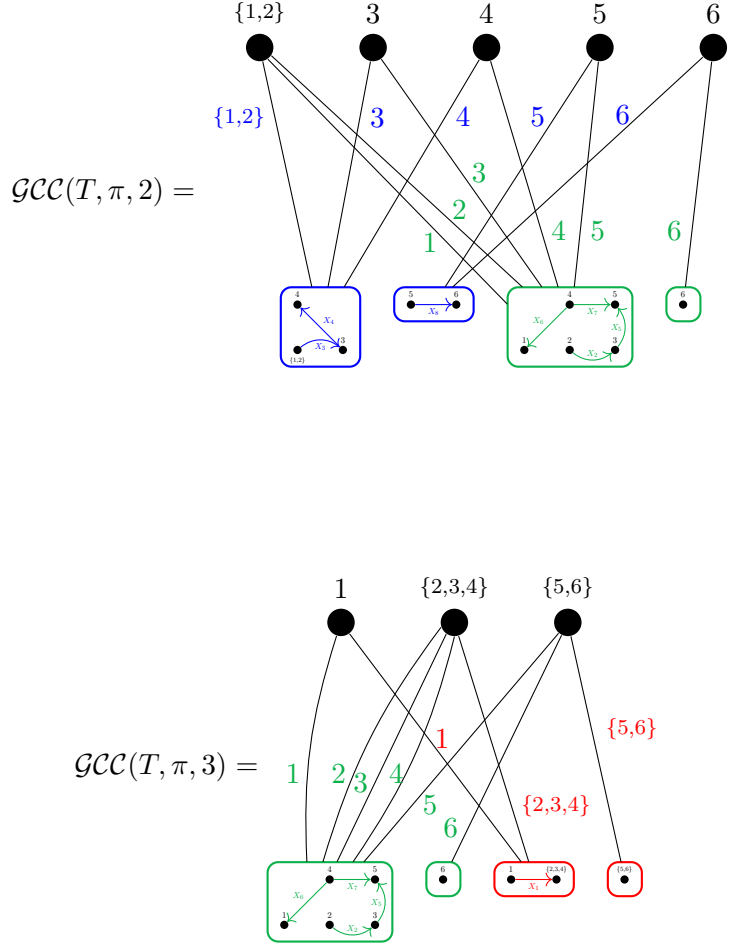
\[
  \GCC(T,\pi,2) =
  \begin{tikzpicture}[baseline]
    \begin{scope}[yshift=-2cm,xshift=.5cm]
      \begin{scope}[scale=.3,transform shape]
        \node[circle,fill,label=below:{\scriptsize{\{1,2\}}}] (1) at (0,0) {};
        \node[circle,fill,label=3] (3) at (2,0) {};
        \node[circle,fill,label=4] (4) at (0,2) {};

        \path[->,color=blue,in=140,out=40] (1) edge node[auto,swap] {$X_3$} (3);
        \path[->,color=blue] (3) edge node[auto,swap] {$X_4$} (4);

        \node [draw,rectangle,text width=35mm,blue, minimum height=35mm, inner sep=0pt,above right,rounded corners,thick] (b1) at (-.75,-.75) {};
      \end{scope}
      \begin{scope}[scale=.3,transform shape,xshift=5cm,yshift=2cm]
        \node[circle,fill,label=5] (5) at (0,0) {};
        \node[circle,fill,label=6] (6) at (2,0) {};

        \path[->,color=blue] (5) edge node[auto,swap] {$X_8$} (6);

        \node [draw,rectangle,text width=35mm,blue, minimum height=15mm, inner sep=0pt,above right,rounded corners,thick] (b2) at (-.75,-.75) {};
      \end{scope}

      \begin{scope}[scale=.3,transform shape,xshift=10cm]
        \node[circle,fill,label=1] (1) at (0,0) {};
        \node[circle,fill,label=2] (2) at (2,0) {};
        \node[circle,fill,label=3] (3) at (4,0) {};
        \node[circle,fill,label=4] (4) at (2,2) {};
        \node[circle,fill,label=5] (5) at (4,2) {};

        \path[->,color=green,in=220,out=320] (2) edge node[auto,swap,above] {$X_2$} (3);
        \path[->,color=green,out=50,in=310] (3) edge node[auto,swap,left] {$X_5$} (5);
        \path[->,color=green] (4) edge node[auto,swap] {$X_6$} (1);
        \path[->,color=green] (4) edge node[auto,swap] {$X_7$} (5);

        \node [draw,green,rectangle,text width=55mm, minimum height=35mm, inner sep=0pt,above right,rounded corners,thick] (g1) at (-.75,-.75) {};
      \end{scope}

      \begin{scope}[scale=.3,transform shape,xshift=17cm,yshift=2cm]
        \node[circle,fill,label=6] (6) at (0,0) {};

        \node [green,draw,rectangle,text width=15mm, minimum height=15mm, inner sep=0pt,above right,rounded corners,thick] (g2) at (-.75,-.75) {};
      \end{scope}
    \end{scope}

    \node[circle,fill,label={\scriptsize{\{1,2\}}}] (1) at (0,2) {};
    \node[circle,fill,label=3] (3) at (1.5,2) {};
    \node[circle,fill,label=4] (4) at (3,2) {};
    \node[circle,fill,label=5] (5) at (4.5,2) {};
    \node[circle,fill,label=6] (6) at (6,2) {};

    \draw (b1) -- (1) node [pos=.75,left,blue] {\scriptsize{\{1,2\}}} ;
    \draw (b1) -- (3) node [pos=.75,right,blue] {3} ;
    \draw (b1) -- (4) node [pos=.75,right,blue] {4} ;
    \draw (b2) -- (5) node [pos=.75,left,blue] {5} ;
    \draw (b2) -- (6) node [pos=.75,left,blue] {6} ;

    \draw (g1.160) -- (1.305) node [pos=.25,left,green] {1} ;
    \draw (g1.140) -- (1.325) node [pos=.25,above,green] {2} ;
    \draw (g1) -- (3) node [pos=.5,right,green] {3} ;
    \draw (g1) -- (4) node [pos=.25,right,green] {4} ;
    \draw (g1) -- (5) node [pos=.25,right,green] {5} ;
    \draw (g2) -- (6) node [pos=.25,left,green] {6} ;
  \end{tikzpicture}
\]

\[
  \GCC(T,\pi,3) =
  \begin{tikzpicture}[baseline]
    \begin{scope}[yshift=-2cm,xshift=-.75cm]
      \begin{scope}[scale=.3,transform shape]
        \node[circle,fill,label=1] (1) at (0,0) {};
        \node[circle,fill,label=2] (2) at (2,0) {};
        \node[circle,fill,label=3] (3) at (4,0) {};
        \node[circle,fill,label=4] (4) at (2,2) {};
        \node[circle,fill,label=5] (5) at (4,2) {};

        \path[->,color=green,in=220,out=320] (2) edge node[auto,swap,above] {$X_2$} (3);
        \path[->,color=green,out=50,in=310] (3) edge node[auto,swap,left] {$X_5$} (5);
        \path[->,color=green] (4) edge node[auto,swap] {$X_6$} (1);
        \path[->,color=green] (4) edge node[auto,swap] {$X_7$} (5);

        \node [draw,green,rectangle,text width=55mm, minimum height=35mm, inner sep=0pt,above right,rounded corners,thick] (g1) at (-.75,-.75) {};
      \end{scope}

      \begin{scope}[scale=.3,transform shape,xshift=7cm,yshift=2cm]
        \node[circle,fill,label=6] (6) at (0,0) {};

        \node [green,draw,rectangle,text width=15mm, minimum height=15mm, inner sep=0pt,above right,rounded corners,thick] (g2) at (-.75,-.75) {};
      \end{scope}

      \begin{scope}[scale=.3,transform shape,xshift=10cm,yshift=2cm]
        \node[circle,fill,label=1] (1) at (0,0) {};
        \node[circle,fill,label=above:{\scriptsize{\{2,3,4\}}}] (2) at (2,0) {};

        \path[->,color=red] (1) edge node[auto,swap] {$X_1$} (2);

        \node [draw,red,rectangle,text width=35mm, minimum height=15mm, inner sep=0pt,above right,rounded corners,thick] (r1) at (-.75,-.75) {};
      \end{scope}

      \begin{scope}[scale=.3,transform shape,xshift=15cm,yshift=2cm]
        \node[circle,fill,label={\scriptsize{\{5,6\}}}] (5) at (0,0) {};

        \node [draw,red,rectangle,text width=15mm, minimum height=15mm, inner sep=0pt,above right,rounded corners,thick] (r2) at (-.75,-.75) {};
      \end{scope}
    \end{scope}

    \node[circle,fill,label=1] (1) at (0,2) {};
    \node[circle,fill,label={\scriptsize{\{2,3,4\}}}] (2) at (1.5,2) {};
    \node[circle,fill,label={\scriptsize{\{5,6\}}}] (5) at (3,2) {};

    \draw (g1.120) to[bend left=10] (1) node [pos=.25,left=.4cm,green] {1} ;
    \draw (g1.100) to[bend left=10] (2.200) node [pos=.25,left=-.2cm,green] {2} ;
    \draw (g1.80) -- (2.235) node [pos=.35,green,left=-.04cm] {3} ;
    \draw (g1.60) to[bend right=10] (2.270) node [pos=.25,right=.5cm,green] {4} ;
    \draw (g1) -- (5) node [pos=.25,right,green] {5} ;
    \draw (g2) -- (5) node [pos=.15,left,green] {6} ;

    \draw (r1) -- (1) node [pos=.5,red,above right=-.2cm] {1};
    \draw (r1) -- (2) node [pos=.25,right,red] {\scriptsize{\{2,3,4\}}};
    \draw (r2) -- (5) node [pos=.5,red,above right] {\scriptsize{\{5,6\}}};
  \end{tikzpicture}
\]

Notice how the edges adjacent to $\set{1,2}$ behave in $\GCC(T,\pi,2)$.
Since $1, 2 \in V(T)$ are identified in $T_{\pi, \cb}$ there is only one edge to the blue component; since they are not identified in $T_{\pi, \cg}$ there are two edges to the green component.

Here, neither $\GCC(T,\pi,2)$ nor $\GCC(T, \pi, 3)$ is a tree; although we didn't draw it, neither is $\GCC(T,\pi,1)$.
Contrarily, it can be checked that if $\sigma_1 = \set{\set{1,2,3,4,5}, \set6}$, $\sigma_2 = \set{\set{1,2,3,4}, \set{5},\set{6}}$, and $\sigma_3 = \set{\set{1,2,3,4}, \set{5,6}}$ then $\GCC(T,\sigma,1)$, $\GCC(T, \sigma,2)$, and $\GCC(T,\sigma,3)$ are all trees.

Finally, let us consider an example application of Lemma~\ref{lem:inducedgccpath}.
Consider the walk ${\color{red} X_1}, {\color{blue}X_3}, {\color{blue}X_4}, {\color{green}X_7}, {\color{blue}X_8}$ in $T$; we will examine the walk it induces in both $\GCC(T,\pi,2)$ and $\GCC(T,\pi,3)$.

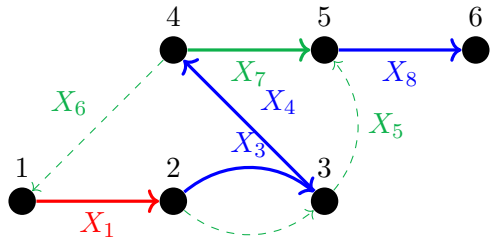
\[\begin{tikzpicture}
    \node[circle,fill,label=1] (1) at (0,0) {};
    \node[circle,fill,label=2] (2) at (2,0) {};
    \node[circle,fill,label=3] (3) at (4,0) {};
    \node[circle,fill,label=4] (4) at (2,2) {};
    \node[circle,fill,label=5] (5) at (4,2) {};
    \node[circle,fill,label=6] (6) at (6,2) {};

    \path[->,color=red,very thick] (1) edge node[auto,swap] {$X_1$} (2);
    \path[->,color=green,dashed,in=220,out=320] (2) edge node[auto,swap] {$X_2$} (3);
    \path[->,color=blue,very thick,in=140,out=40] (2) edge node[auto,swap,above] {$X_3$} (3);
    \path[->,color=blue,very thick] (3) edge node[auto,swap] {$X_4$} (4);
    \path[->,color=green,dashed,out=50,in=310] (3) edge node[auto,swap] {$X_5$} (5);
    \path[->,color=green,dashed] (4) edge node[auto,swap] {$X_6$} (1);
    \path[->,color=green,very thick] (4) edge node[auto,swap] {$X_7$} (5);
    \path[->,color=blue,very thick] (5) edge node[auto,swap] {$X_8$} (6);
\end{tikzpicture}\]

In $\GCC(T,\pi,2)$ the path begins at $\set{1,2}$, and proceeds to $3$ before $4$.
The edge corresponding to ${\color{blue}3}$ is traversed twice, once in each direction.
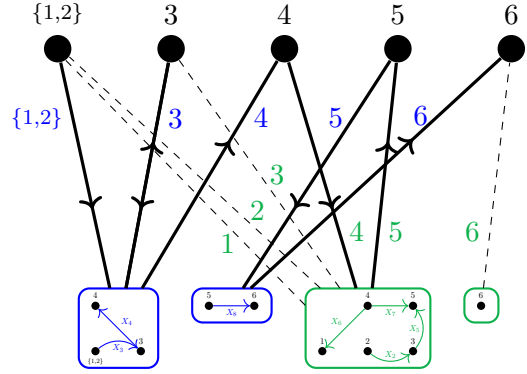
\[\begin{tikzpicture}[baseline]
    \begin{scope}[yshift=-2cm,xshift=.5cm]
      \begin{scope}[scale=.3,transform shape]
        \node[circle,fill,label=below:{\scriptsize{\{1,2\}}}] (1) at (0,0) {};
        \node[circle,fill,label=3] (3) at (2,0) {};
        \node[circle,fill,label=4] (4) at (0,2) {};

        \path[->,color=blue,in=140,out=40] (1) edge node[auto,swap] {$X_3$} (3);
        \path[->,color=blue] (3) edge node[auto,swap] {$X_4$} (4);

        \node [draw,rectangle,text width=35mm,blue, minimum height=35mm, inner sep=0pt,above right,rounded corners,thick] (b1) at (-.75,-.75) {};
      \end{scope}
      \begin{scope}[scale=.3,transform shape,xshift=5cm,yshift=2cm]
        \node[circle,fill,label=5] (5) at (0,0) {};
        \node[circle,fill,label=6] (6) at (2,0) {};

        \path[->,color=blue] (5) edge node[auto,swap] {$X_8$} (6);

        \node [draw,rectangle,text width=35mm,blue, minimum height=15mm, inner sep=0pt,above right,rounded corners,thick] (b2) at (-.75,-.75) {};
      \end{scope}

      \begin{scope}[scale=.3,transform shape,xshift=10cm]
        \node[circle,fill,label=1] (1) at (0,0) {};
        \node[circle,fill,label=2] (2) at (2,0) {};
        \node[circle,fill,label=3] (3) at (4,0) {};
        \node[circle,fill,label=4] (4) at (2,2) {};
        \node[circle,fill,label=5] (5) at (4,2) {};

        \path[->,color=green,in=220,out=320] (2) edge node[auto,swap,above] {$X_2$} (3);
        \path[->,color=green,out=50,in=310] (3) edge node[auto,swap,left] {$X_5$} (5);
        \path[->,color=green] (4) edge node[auto,swap] {$X_6$} (1);
        \path[->,color=green] (4) edge node[auto,swap] {$X_7$} (5);

        \node [draw,green,rectangle,text width=55mm, minimum height=35mm, inner sep=0pt,above right,rounded corners,thick] (g1) at (-.75,-.75) {};
      \end{scope}

      \begin{scope}[scale=.3,transform shape,xshift=17cm,yshift=2cm]
        \node[circle,fill,label=6] (6) at (0,0) {};

        \node [green,draw,rectangle,text width=15mm, minimum height=15mm, inner sep=0pt,above right,rounded corners,thick] (g2) at (-.75,-.75) {};
      \end{scope}
    \end{scope}

    \node[circle,fill,label={\scriptsize{\{1,2\}}}] (1) at (0,2) {};
    \node[circle,fill,label=3] (3) at (1.5,2) {};
    \node[circle,fill,label=4] (4) at (3,2) {};
    \node[circle,fill,label=5] (5) at (4.5,2) {};
    \node[circle,fill,label=6] (6) at (6,2) {};

    \draw [dashed] (b1) -- (1) node [pos=.75,left,blue] {\scriptsize{\{1,2\}}} ;
    \path [dashed] (b1) -- (3) node [pos=.75,right,blue] {3} ;
    \draw [dashed] (b1) -- (4) node [pos=.75,right,blue] {4} ;
    \draw [dashed] (b2) -- (5) node [pos=.75,left,blue] {5} ;
    \draw [dashed] (b2) -- (6) node [pos=.75,left,blue] {6} ;

    \draw [dashed] (g1.160) -- (1.305) node [pos=.25,left,green] {1} ;
    \draw [dashed] (g1.140) -- (1.325) node [pos=.25,above,green] {2} ;
    \draw [dashed] (g1) -- (3) node [pos=.5,right,green] {3} ;
    \draw [dashed] (g1) -- (4) node [pos=.25,right,green] {4} ;
    \draw [dashed] (g1) -- (5) node [pos=.25,right,green] {5} ;
    \draw [dashed] (g2) -- (6) node [pos=.25,left,green] {6} ;

    \begin{scope}[very thick,decoration={markings,mark=at position 0.65 with {\arrow{>}}}]
      \draw[postaction={decorate}] (1) -- (b1);
      \draw[postaction={decorate}] (b1) -- (3);
      \draw[postaction={decorate}] (3) -- (b1);
      \draw[postaction={decorate}] (b1) -- (4);
      \draw[postaction={decorate}] (4) -- (g1);
      \draw[postaction={decorate}] (g1) -- (5);
      \draw[postaction={decorate}] (5) -- (b2);
      \draw[postaction={decorate}] (b2) -- (6);
    \end{scope}
  \end{tikzpicture}
\]

Although the path in $T$ passes from $2$ to $3$ to $4$, it does so along blue edges which are not reflected in $\GCC(T,\pi,3)$, even though there is also (for example) a green path from $2$ to $3$.
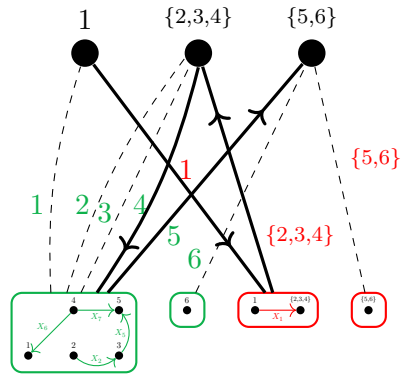
\[\begin{tikzpicture}[baseline]
    \begin{scope}[yshift=-2cm,xshift=-.75cm]
      \begin{scope}[scale=.3,transform shape]
        \node[circle,fill,label=1] (1) at (0,0) {};
        \node[circle,fill,label=2] (2) at (2,0) {};
        \node[circle,fill,label=3] (3) at (4,0) {};
        \node[circle,fill,label=4] (4) at (2,2) {};
        \node[circle,fill,label=5] (5) at (4,2) {};

        \path[->,color=green,in=220,out=320] (2) edge node[auto,swap,above] {$X_2$} (3);
        \path[->,color=green,out=50,in=310] (3) edge node[auto,swap,left] {$X_5$} (5);
        \path[->,color=green] (4) edge node[auto,swap] {$X_6$} (1);
        \path[->,color=green] (4) edge node[auto,swap] {$X_7$} (5);

        \node [draw,green,rectangle,text width=55mm, minimum height=35mm, inner sep=0pt,above right,rounded corners,thick] (g1) at (-.75,-.75) {};
      \end{scope}

      \begin{scope}[scale=.3,transform shape,xshift=7cm,yshift=2cm]
        \node[circle,fill,label=6] (6) at (0,0) {};

        \node [green,draw,rectangle,text width=15mm, minimum height=15mm, inner sep=0pt,above right,rounded corners,thick] (g2) at (-.75,-.75) {};
      \end{scope}

      \begin{scope}[scale=.3,transform shape,xshift=10cm,yshift=2cm]
        \node[circle,fill,label=1] (1) at (0,0) {};
        \node[circle,fill,label=above:{\scriptsize{\{2,3,4\}}}] (2) at (2,0) {};

        \path[->,color=red] (1) edge node[auto,swap] {$X_1$} (2);

        \node [draw,red,rectangle,text width=35mm, minimum height=15mm, inner sep=0pt,above right,rounded corners,thick] (r1) at (-.75,-.75) {};
      \end{scope}

      \begin{scope}[scale=.3,transform shape,xshift=15cm,yshift=2cm]
        \node[circle,fill,label={\scriptsize{\{5,6\}}}] (5) at (0,0) {};

        \node [draw,red,rectangle,text width=15mm, minimum height=15mm, inner sep=0pt,above right,rounded corners,thick] (r2) at (-.75,-.75) {};
      \end{scope}
    \end{scope}

    \node[circle,fill,label=1] (1) at (0,2) {};
    \node[circle,fill,label={\scriptsize{\{2,3,4\}}}] (2) at (1.5,2) {};
    \node[circle,fill,label={\scriptsize{\{5,6\}}}] (5) at (3,2) {};

    \draw [dashed] (g1.120) to[bend left=10] (1) node [pos=.25,left=.4cm,green] {1} ;
    \draw [dashed] (g1.100) to[bend left=10] (2.200) node [pos=.25,left=-.2cm,green] {2} ;
    \draw [dashed] (g1.80) -- (2.235) node [pos=.35,green,left=-.04cm] {3} ;
    \draw [dashed] (g1.60) to[bend right=10] (2.270) node [pos=.25,right=.5cm,green] {4} ;
    \draw [dashed] (g1) -- (5) node [pos=.25,right,green] {5} ;
    \draw [dashed] (g2) -- (5) node [pos=.15,left,green] {6} ;

    \draw [dashed] (r1) -- (1) node [pos=.5,red,above right=-.2cm] {1};
    \draw [dashed] (r1) -- (2) node [pos=.25,right,red] {\scriptsize{\{2,3,4\}}};
    \draw [dashed] (r2) -- (5) node [pos=.5,red,above right] {\scriptsize{\{5,6\}}};

    \begin{scope}[very thick,decoration={markings,mark=at position 0.8 with {\arrow{>}}}]
      \draw[postaction={decorate}] (1) -- (r1);
      \draw[postaction={decorate}] (r1) -- (2);
      \draw[postaction={decorate}] (2.270) to[bend left=10] (g1.60);
      \draw[postaction={decorate}] (g1) -- (5);
    \end{scope}
  \end{tikzpicture}
\]

\subsection{Summary of Notation}

Here we collect a brief summary of the notation used in Section~\ref{sec:permutationtraffic}.
It is meant to serve as a reminder; the reader should consult the full definitions above.

\begin{center}
\begin{tabular}{|l|p{0.6\textwidth}@{\hspace{2em}}r|}
\hline
  $\cC$ & set of colors labeling the algebras in a graph product &\\
  $\cG$ & $(\cC, \cE)$, a (simple undirected) graph indexing a graph product &\\
  $\cS$ & set of strings labeling tensor factors in $\bigotimes_{\cS} M_N(\bC)$ &\\
  $s\sstrc c$ & the $c$-colored algebra has access to the $s$-labelled tensor factor &\\
   $\cS_c$ & $\set{ s \in \cS \mid s \sstrc c }$ &\\
  $\cC_s$ & $\set{ c \in \cC \mid s \sstrc c }$ &\\
  $(\Sigma_c^{(N)})_{c\in\cC}$ & independent permutation matrices with $\Sigma_c^{(N)}$ uniform among permutations in $\bigotimes_{\cS_c} M_{N}(\bC)$ & (cf. Theorem~\ref{thm: permutation model}) \\
  $\uu{X}$ & $ (\Sigma_{c}^{(N)})^t X \Sigma_c^{(N)} \otimes I_N^{\otimes\cS\setminus \cS_c}$ & (cf. Theorem~\ref{thm: permutation model})\\
  $\tau_N(T)$ & $\frac{1}{N^{\# \Comp(T)}} \sum_{i: V(T) \to [N]} \prod_{e \in E(T)} (X_e)_{i(e_+),i(e_-)}$ & (cf. Definition~\ref{defn:testtrace}) \\
  $\tau_N^\circ(T)$ & $\frac{1}{N^{\# \Comp(T)}} \sum_{i: V(T) \hookrightarrow [N]} \prod_{e \in E(T)} (X_e)_{i(e_+),i(e_-)}$ & (cf. Definition~\ref{defn:testtrace}) \\
  $\mathring{T}$ & $T$ with $\Lambda$-labelled self-loops & (cf. Notation~\ref{not: loop graph}) \\
  $T|_C$ & $T$ with edges not in $\chi^{-1}(C)$ removed & (cf. Notation~\ref{not: rose})\\
  $T|_c$ & $T_{\set{c}}$ & (cf. Notation~\ref{not: rose})\\
  $\rho_s$ & partition of $V(T)$ into connected components of $T|_{\cC\setminus \cC_s}$ & (cf. Notation~\ref{not: rose}) \\
  $\gamma_N(T, \pi)$ & $\displaystyle\frac{1}{N^{\# \mathcal{S}}}
    \sum_{\substack{i: V(T) \to [N]^{\mathcal{S}} \\ \forall s \in \mathcal{S}, \ker(i_s) = \pi_s }}
    \prod_{v \in V(T)} (\Lambda_v)_{i(v),i(v)}
    \prod_{e \in E(T)} (\uu{X_e})_{i(e_+),i(e_-)}$& (cf. (\ref{eq:gamma})) \\
    $\lambda_N(T, \pi)$ & $\displaystyle\frac{1}{N^{\sum_{s \in \cS} \# \pi_s}} \sum_{\substack{i: V(T) \to [N]^{\mathcal{S}} \\ \forall s \in \mathcal{S}, \ker(i_s) = \pi_s }}
  \prod_{v \in V(T)} (\Lambda_v)_{i(v),i(v)}$ & (cf. Notation~\ref{not:lambdaN})\\
  $\omegapic$ & $\bigwedge_{s \in \cS_c} \pi_s$ & (cf. Notation~\ref{not:omegapic})\\
  $\Tpic$ & $T|_{c}/\pi_c$, labelled by $\uu{X}$ & (cf.~Notation~\ref{not:omegapic}) \\
  $\gps$ & $T|_{\cC_s} / \pi_s$, labelled by $\uu{X}$ & (cf. Notation~\ref{not:gpis}) \\
  $\hsc$ & Quotient and restriction map $\Tpic \to T / \pi_s \to \gps$ & (cf. Notation~\ref{not:hsc}) \\
    $\GCC(T, \pi, s)$ & bipartite graph on $\bigsqcup_{c \in \cC_s} \Comp(\Tpic)$ and $V(\gps)$ & (cf. Definition~\ref{def:gcc}) \\
    \hline
\end{tabular}
\end{center}

\newcommand{\etalchar}[1]{$^{#1}$}
\providecommand{\bysame}{\leavevmode\hbox to3em{\hrulefill}\thinspace}
\providecommand{\MR}{\relax\ifhmode\unskip\space\fi MR }
\providecommand{\MRhref}[2]{%
  \href{http://www.ams.org/mathscinet-getitem?mr=#1}{#2}
}
\providecommand{\href}[2]{#2}

%

\end{document}